\documentclass[12pt, reqno]{amsart}
\usepackage{amsmath, amsthm, amscd, amsfonts, amssymb, graphicx, color, mathrsfs}
\usepackage[bookmarksnumbered, colorlinks, plainpages]{hyperref}
\usepackage[all]{xy}
\usepackage{slashed}
\usepackage{tipa}
\usepackage{soul}
\usepackage{cancel}
\usepackage{ulem}

\newtheorem{theorem}{Theorem}[section]
\newtheorem{lemma}[theorem]{Lemma}

\newtheorem{corollary}[theorem]{Corollary}
\theoremstyle{definition}
\newtheorem{definition}[theorem]{Definition}
\newtheorem{example}[theorem]{Example}

\theoremstyle{remark}
\newtheorem{remark}[theorem]{Remark}
\numberwithin{equation}{section}

\begin{document}
\setcounter{page}{1}

\title[ $L^p$-bounds on graded Lie groups  ]{$L^p$-bounds for pseudo-differential operators on graded Lie groups }

\author[D. Cardona]{Duv\'an Cardona}
\address{
  Duv\'an Cardona:
  \endgraf
  Department of Mathematics: Analysis, Logic and Discrete Mathematics
  \endgraf
  Ghent University, Belgium
  \endgraf
  {\it E-mail address} {\rm duvanc306@gmail.com, duvan.cardonasanchez@ugent.be}
  }
  
  \author[J. Delgado]{Julio Delgado}
\address{
  Julio Delgado:
  \endgraf
  Departmento de Matem\'aticas
  \endgraf
  Universidad del Valle
  \endgraf
  Cali-Colombia
    \endgraf
    {\it E-mail address} {\rm delgado.julio@correounivalle.edu.co}
  }

\author[M. Ruzhansky]{Michael Ruzhansky}
\address{
  Michael Ruzhansky:
  \endgraf
  Department of Mathematics: Analysis, Logic and Discrete Mathematics
  \endgraf
  Ghent University, Belgium
  \endgraf
 and
  \endgraf
  School of Mathematical Sciences
  \endgraf
  Queen Mary University of London
  \endgraf
  United Kingdom
  \endgraf
  {\it E-mail address} {\rm michael.ruzhansky@ugent.be}
  }



\thanks{The first author was supported  by the FWO Odysseus Project. The second author was supported by the Leverhulme Grant RPG-2017-151 and by Grant CI71234 Vic. Inv. Universidad del Valle. The third author was supported in parts by the FWO Odysseus Project and also  by the Leverhulme Grant RPG-2017-151 and  EPSRC grant EP/R003025.}

\begin{abstract} In this work
we obtain  sharp  $L^p$-estimates  for  pseudo-differential operators on arbitrary  graded Lie groups.
 The results are presented within the setting of the global symbolic calculus on graded Lie groups  by using the Fourier analysis associated to every graded Lie group which extends the usual one due to H\"ormander on $\mathbb{R}^n$. The main result extends the classical Fefferman's  sharp theorem on the $L^p$-boundedness of pseudo-differential operators  for H\"ormander classes on $\mathbb{R}^n$ to general graded Lie groups, also adding   the borderline $\rho=\delta$ case. 
\end{abstract} \maketitle

\tableofcontents
\allowdisplaybreaks
\section{Introduction}
 The investigation of the $L^p$ boundedness of pseudo-differential operators is a crucial task for a large 
 variety of problems in mathematical analysis and its applications, mainly due to its consequences for the  regularity, approximation and existence of solutions on $L^p$-Sobolev spaces. 
 There is an extensive literature on the subject, in particular, devoted to operators associated with symbols belonging to the H\"ormander classes $S^m_{\rho,\delta}(\mathbb{R}^n\times \mathbb{R}^n),$ (see for instance, J.J. Kohn and L. Nirenberg \cite{KohnNirenberg1965}, L. H\"ormander \cite{HormanderBook34}  and C. Fefferman \cite{Fefferman1973}). Applications of the $L^p$-estimates for H\"ormander classes to PDE, particularly, to the  $L^p$-theory of linear and non-linear elliptic and hypoelliptic equations can be found e.g. in  the book of  M. Taylor \cite{Taylorbook1981} and the seminal  volumes III and IV of  H\"ormander's book \cite{HormanderBook34}. Although the classical $L^p$-estimates are established for smooth symbols, a careful analysis could provide $L^p$-estimates for versions of H\"ormander classes with symbols of limited regularity allowing to apply these results to some (possibly) non-linear PDE whose coefficients could have limited regularity (see e.g. M.  Taylor \cite{Taylorbook1981}, J. M. Bony \cite{Bony81,Bony84}, G. Bourdaud \cite{Bourdaud88}, H. Kumano-Go and M. Nagase \cite{Kumano-goNagase} and R. Coifman and Y. Meyer \cite{CoifmanMeyer78}).      The purpose of this paper is to prove sharp $L^p$-estimates for pseudo-differential operators on graded Lie groups. Our main goal is to extend a classical and sharp result by C. Fefferman \cite{Fefferman1973} and to provide a critical order for the $L^p$-boundendess of pseudo-differential operators on graded Lie groups based on the quantization procedure developed by the third author and V. Fischer in \cite{FischerRuzhanskyBook2015}.
 
 As it was reviewed in \cite[page 16]{FischerRuzhanskyBook2015}, the  analysis on homogeneous Lie groups and also on other types
of Lie groups has received another boost with newly found applications and further
advances in many topics in the last years. The particular case of graded Lie groups appears naturally in the analysis on nilpotent Lie groups and smooth  manifolds providing an
abstract setting for many notions of Euclidean analysis. The most natural group appearing in this context as the less non-commutative nilpotent group, is the Heisenberg group.  Its study and its applications are a very active field of research (see e.g. R. Howe \cite{Howe}) due to its role in the  interplay between  analysis, geometry,  representation theory and  sub-Riemannian geometry aspects. 

Nilpotent Lie groups by themselves appear  as local models  in the works on the construction of parametrices for the Kohn-Laplacian (the Laplacian associated to the tangential CR complex on the boundary $X$ of a strictly pseudo-convex domain) and for other differential operators. The corresponding parametrices and subsequent subelliptic estimates
have been obtained by Folland and Stein in \cite{FollandStein74} by first establishing a version
of the results for a family of sub-Laplacians on the Heisenberg group, and then
for the Kohn-Laplacian, by replacing $X$ locally by the Heisenberg group. These
ideas soon led to powerful generalisations. Indeed, the general techniques for approximating differential operators on a manifold by left-invariant operators on a nilpotent Lie group
have been developed by Rothschild and Stein in \cite{RothschildStein76}. 
  Since our results herein absorb those of Fefferman \cite{Fefferman1973} which are sharp,  we recall the following sharp $L^p$-estimate  due to Charles Fefferman.  Because the $L^p$-boundedness of operators in the class  $S^0_{1,0}(\mathbb{R}^n\times {\mathbb{R}^n}),$ was known previously to  Fefferman's work (see Kohn and Nirenberg \cite{KohnNirenberg1965}), the reference \cite{Fefferman1973}  only considers the case $0<\rho<1$.
\begin{theorem}\label{FT}
Let $A:C^\infty(\mathbb{R}^n)\rightarrow\mathscr{D}'(\mathbb{R}^n)$ be a pseudo-differential operator with symbol $\sigma\in S^{-m}_{\rho,\delta}(\mathbb{R}^n\times {\mathbb{R}}^n ),$  $0\leq \delta<\rho< 1.$ Then,
\begin{itemize}
    \item{\textnormal{(a)}} if $m=\frac{n(1-\rho)}{2},$ then $A$ extends to a bounded operator from $L^\infty(\mathbb{R}^n)$ to $ BMO(\mathbb{R}^n)$ and also admits a bounded extension from the Hardy space $H^1(\mathbb{R}^n)$ to $L^1(\mathbb{R}^n)$. Moreover, for all $1<p<\infty,$ $A$ admits a bounded extension on $L^p(\mathbb{R}^n).$
   \item{\textnormal{(b)}} If $m\geq m_{p}:= n(1-\rho)\left|\frac{1}{p}-\frac{1}{2}\right|,$ then $A$ extends to a bounded operator on $ L^p(\mathbb{R}^n),$ for $1<p<\infty.$ 
\end{itemize}
\end{theorem} 
We should recall that the condition   
 $\delta<\rho$ in Fefferman's theorem can be improved allowing $\delta\leq \rho.$ This result is due to  C. Z. Li, and R. H. Wang,  \cite{RouhuaiChengzhang} (see also e.g. Miyachi \cite{Miyachi88} or L. Wang \cite[page 15]{LW}). As a consequence of the approach that we use in the proof of our main theorem (Theorem \ref{MainTheorem}), and considering that our methods are different to those employed by C. Z. Li, and R. H. Wang, we will present (in the case of $G=\mathbb{R}^n$ in Theorem \ref{MainTheorem}), an alternative proof for the extension of Fefferman's theorem for $\delta\leq \rho,$ up to by conditions of limited regularity. After presenting our main theorem,  we will return to this discussion.
The main point in Theorem \ref{FT} is the $L^\infty$-$BMO$ boundedness. From there, by using the classical Fefferman-Stein duality $(H^1)'=BMO,$ the real and the complex interpolation, we can deduce the other parts of the theorem. The Fefferman's proof of the $L^\infty$-$BMO$ estimate relies on the analysis of  pieces of the symbol arising from a partition of unity in the spirit of the Littlewood-Paley theory. Fefferman's Theorem \ref{FT} is sharp in the following sense. If $m<m_p,$  and 
\begin{equation}\label{oscillatory}
    \sigma_{\rho,m}(\xi)={e^{i(1+|\xi|^2)^{ \frac{1-\rho}{2} }}}{(1+|\xi|^2)^{-\frac{m}{2}}}\in S^{-m}_{\rho,0}(\mathbb{R}^n\times {\mathbb{R}}^n ),
\end{equation}then $A=\sigma_{\rho,m}(D),$ that is the Fourier multiplier with symbol $\sigma_{\rho,m},$ is unbounded on $L^p(\mathbb{R}^n)$ for all $1<p<\infty.$ For $p=1,$ and $m=m_1:=\frac{n(1-\rho)}{2},$  $\sigma_{\rho,m}(D)$ is unbounded from  $L^1(\mathbb{R}^n)$ into $L^1(\mathbb{R}^n).$ However, in view of  part \textnormal{(a)} of Theorem \ref{FT},  $\sigma_{\rho,m}(D)$ is bounded from $H^1(\mathbb{R}^n)$ to $L^1(\mathbb{R}^n)$ and from   $L^\infty(\mathbb{R}^n)$ to $ BMO(\mathbb{R}^n).$ This counterexample is due to Hardy-Littlewood, Hirschman and Wainger. Part \textnormal{(b)}  was proved in the noncritical case   by Hirschman and Wainger \cite{Hirschman1956,Wainger1965} (for Fourier multipliers) and by H\"ormander \cite{Hormander1967} for $m>m_{p}$. 

Due to the boundedness properties  of $\sigma_{\rho,m}$ in \eqref{oscillatory} on  $L^p(\mathbb{R}^n),$ the order $m_p$ is  the critical decay order for the $L^p$-boundedness of pseudo-differential operators in the H\"ormander classes. Similar  critical orders also appear in the $L^p$-boundedness of pseudo-differential operators associated with the Weyl-H\"ormander classes $S(m,g)$ on $\mathbb{R}^n$ and H\"ormander classes on compact Lie groups. We refer the reader to the works  to \cite{Delgado2006,Delgado2015} and \cite{DelgadoRuzhansky2019}  where Theorem \ref{FT} was extended for the  $S(m,g)$ classes and also for the global  classes on compact Lie groups, respectively.

Our main theorem is the following extension of Theorem \ref{FT} to an arbitrary graded Lie group $G$ of homogeneous dimension $Q$, where we found the critical order $m_p:=Q(1-\rho)\left|\frac{1}{p}-\frac{1}{2}\right|,$ for the $L^p$-boundedness of the  H\"ormander classes for the quantization process developed in \cite{FischerRuzhanskyBook2015}. 
\begin{theorem}\label{MainTheorem}
Let $G$ be a graded Lie group of homogeneous dimension $Q.$ Let $A:C^\infty(G)\rightarrow\mathscr{D}'(G)$ be a pseudo-differential operator with symbol $\sigma\in S^{-m}_{\rho,\delta}(G\times \widehat{G} ),$ $0\leq \delta\leq \rho\leq 1,$ $\delta\neq 1.$ Then,
\begin{itemize}
    \item{\textnormal{(a)}} if ${m=\frac{Q  (1-\rho)  }{2}},$  then $A$ extends to a bounded operator from $L^\infty(G)$ to $ BMO(G),$ from the Hardy space $H^1(G)$ to $L^1(G),$ and  from $L^p(G)$ to $L^p(G)$ for all $1< p<\infty.$ 
    \item{\textnormal{(b)}} If $m\geq m_{p}:= Q(1-\rho)\left|\frac{1}{p}-\frac{1}{2}\right|,$ $1<p<\infty,$ then $A$ extends to a bounded operator from $ L^p(G)$ into $ L^p(G).$  
\end{itemize}
\end{theorem} Now, we briefly discuss our main result. 
\begin{itemize}
    \item In the case of $G=\mathbb{R}^n,$ Theorem  \ref{MainTheorem}  is precisely the sharp Fefferman theorem (Theorem \ref{FT}). In particular, for the $L^\infty$-$BMO$ boundedness of operators with order $-Q(1-\rho)/2,$ we extend Feferman's theorem imposing the condition $0\leq \delta\leq \rho\leq 1,$ $\delta\neq 1,$ allowing in this case $\rho=\delta,$ on general graded Lie groups.
    \item On $\mathbb{R}^n,$ that the condition $\delta< \rho,$ can be relaxed to $\delta\leq \rho,$ was first observed by   C. Z. Li, and R. H. Wang, \cite{RouhuaiChengzhang}. Because in our proof on general graded Lie groups we use a different approach to that used in \cite{RouhuaiChengzhang}, our analysis in particular gives a new proof of this fact, up to by conditions of limited regularity (see Remark \ref{remarkestimate}).  
    \item In  the critical case, $\rho=\delta=0,$ Theorem \ref{MainTheorem} provides the $L^\infty(G)$-$BMO(G),$ the $H^1(G)$-$L^1(G),$ and the $L^p(G)$-$L^p(G)$-boundedness for operators associated to the class $ S^{-\frac{Q}{2}}_{0,0}(G\times \widehat{G} ).$  For $G=\mathbb{R}^n,$  this  result  is due to Coiffman and Meyer \cite{CoifmanMeyer78}. With the notation of Remark \ref{remarkoffinite}, for the case  $\rho=\delta=0,$   we impose  difference conditions up to order  $ rp_0,$ and derivatives in the spatial variable up to order $ r\nu+[\frac{Q}{2}],$  where $\nu$ is the degree of homogeneity of the Rockland operator $\mathcal{R}$ fixed in \eqref{seminorm},  $p_0/2$ is the smallest positive integer divisible by the weights $\nu_1,\nu_2,\cdots ,\nu_n,$ (see Section \ref{Preliminaries}) and $r\in\mathbb{N}_0$ is the smallest integer such that $rp_0>Q+1.$
    \item If $p=2,$ Part (b) of Theorem \ref{MainTheorem}, provides the $L^2(G)$-boundedness for pseudo-differential operators associated to the class $S^{0}_{\rho,\rho}(G\times \widehat{G} ),$ $0\leq \rho<1.$ So, we re-obtain the Calder\'on-Vaillacourt theorem  proved in  the graded setting   in   Proposition 5.7.14 of \cite{FischerRuzhanskyBook2015}. For $G=\mathbb{R}^n,$ it is well know that the Calder\'on-Vaillancourt theorem \cite{CalderonVaillancourt71} is sharp, indeed, for every $1<p<\infty,$ there exists an operator with symbol in   $S^{0}_{1,1}(\mathbb{R}^n\times \mathbb{R}^n ),$ which extends to an unbounded operator on $L^p(\mathbb{R}^n),$ (see e.g. Wang \cite[page 14]{LW} and Taylor \cite{Taylorbook1981}).
    \item Although, some  operators in the class  $S^{0}_{1,1}(\mathbb{R}^n\times \mathbb{R}^n ),$ are unbounded on $L^p(G),$ it is well known that the class $ S^{0}_{1,\delta}(\mathbb{R}^n\times \mathbb{R}^n ),$ $0\leq \delta<1,$ provides operators  admitting a  bounded extension  on $L^p(\mathbb{R}^n),$ for all $1<p<\infty,$ (see e.g. Taylor \cite{Taylorbook1981}).  If in Theorem \ref{oscillatory} we consider $\rho=1,$ and $0\leq \delta<1,$ we observe that the class  $ S^{0}_{1,\delta}(G\times \widehat{G} ),$ $0\leq \delta<1,$ begets operators  admitting   bounded extensions  on $L^p(G),$ for all $1<p<\infty.$ To do this, we will use the $L^2$-boundedeness of pseudo-differential operators in the class $S^0_{1,\delta}(G\times \widehat{G})$ proved in \cite{FischerRuzhanskyBook2015} (see Theorem \ref{1delta}). A more general condition (that indeed, is an extension of the H\"ormander-Mihlin condition on $\mathbb{R}^n$) for the $L^p$-boundedness and the weak type (1,1) of Fourier multipliers on graded Lie groups has been established in \cite{FischerRuzhansky2014}. 
    \item For $G=\mathbb{H}^n,$  that is, the Heisenberg group,   the H\"ormander classes on the Heisenberg group $S^{m}_{\rho,\delta}(\mathbb{H}^n\times \widehat{\mathbb{H}}^n )$  and the Shubin calculus are related (see \cite[Chapter 6]{FischerRuzhanskyBook2015}). We will discuss this connection in Remark \ref{MainRemark} and the analogy of Theorem \ref{MainTheorem} in terms of the Shubin classes will be presented in Corollary \ref{MainCorollary}. 
    \item Although we present our results in terms of smooth symbols, we only need finite  regularity in $x,$ and also a finite number of difference conditions for the Fourier variables in the unitary  dual $ \widehat{G},$ (see e.g. Eq.  \eqref{LBMO} and \eqref{H1L1}).
\end{itemize}
In the setting of graded Lie groups, the $L^p$-boundedness of the global pseudo-differential calculus for the classes $S^m_{\rho,\delta}(G\times \widehat{G})$ have been investigated by the third author and V. Fischer with the    Calder\'on-Vaillancourt theorem in \cite{FischerRuzhanskyBook2015} regarding  the $L^2(G)$-boundedness for operators with symbol in the class $S^0_{\rho,\rho}(G\times \widehat{G}),$ and  the $L^p$-boundedness of operators in the class  $S^0_{1,0}(G\times \widehat{G}),$ (see \cite[Corollary 5.7.4]{FischerRuzhanskyBook2015}) which corresponds with the case $\rho=1$ in Theorem \ref{MainTheorem}. The H\"ormander-Mihlin theorem  for Fourier multipliers  on graded Lie groups has been established also in  \cite{FischerRuzhansky2014}. For the case of compact Lie groups, an analogy of Theorem \ref{MainTheorem} has been established for $(\rho,\delta)$-H\"ormander classes by the second and third author in \cite{DelgadoRuzhansky2019} while the H\"ormander-Mihlin condition for Fourier multipliers was established by the third author and J. Wirth in \cite{RuzhanskyWirth2015}. The H\"ormander-Mihlin theorem for spectral multipliers on  Lie groups of polynomial growth appears e.g. in the work of Alexopoulos \cite{Alexopoulos1994}, and other conditions in the nilpotent setting for spectral multipliers of the sub-Laplacian or Rockland operators in the graded setting can be found in  Christ \cite{Christ91}, Christ and M\"uller \cite{ChristMuller}, De Michele and  Mauceri \cite{DeMicheleMauceri,DeMicheleMauceri2}, Martini \cite{Martini} and Martini and M\"uller \cite{MartiniThesis,MartiniMuller}, where many of these works  can be considered as non-commutative extensions of the classical $L^p$-H\"ormander-Mihlin-Marcinkiewicz theorems (see e.g. H\"ormander \cite{Hormander1960}, Marcinkiewicz \cite{Marcinkiewicz1939} and Mihlin \cite{Mihlin1956} and the recent revision on the subject by  Grafakos and Slav\'ikov\'a \cite{GrafakosSlavikova2019I,GrafakosSlavikova2019II}). 

We also state the corresponding $L^p$-Sobolev and Besov bounds that can be deduced from Theorem \ref{MainTheorem} (see Theorem \ref{Sobolevtheorem}) and we also study the boundedness of local versions of global H\"ormander classes, on local Sobolev spaces on the group. We also compare the boundededness on local Sobolev spaces on $\mathbb{R}^n,$ obtaining both situations, loss of regularity and gain of regularity (see Remark \ref{losswin}). 

This paper is organized as follows. In Section \ref{Preliminaries} we will present some preliminaries on homogeneous Lie groups and the quantization process of pseudo-differential operators developed in \cite{FischerRuzhanskyBook2015}. Finally, in Section \ref{sec3}, we prove our main theorem and some of its consequences for the boundedness of operators on Sobolev and Besov spaces. 

\section{Preliminaries and global  quantization  on graded Lie groups}\label{Preliminaries}
 The notation and terminology of this paper on the analysis of homogeneous Lie groups are mostly taken 
from Folland and Stein \cite{FollandStein1982}. For the theory of pseudo-differential operators we will follow the setting developed in \cite{FischerRuzhanskyBook2015} through  the notion of (operator-valued) global symbols. If $E,F$ are Hilbert spaces,  $\mathscr{B}(E,F)$ denotes the algebra of bounded linear operators from $E$ to $F,$ and also we will write $\mathscr{B}(E)=\mathscr{B}(E,E).$
    \subsection{Homogeneous and graded Lie groups} 
    Let $G$ be a homogeneous Lie group. This means that $G$ is a connected and simply connected Lie group whose Lie algebra $\mathfrak{g}$ is endowed with a family of dilations $D_{r}^{\mathfrak{g}},$ $r>0,$ which are automorphisms on $\mathfrak{g}$  satisfying the following two conditions:
\begin{itemize}
\item For every $r>0,$ $D_{r}^{\mathfrak{g}}$ is a map of the form
$$ D_{r}^{\mathfrak{g}}=\textnormal{Exp}(\ln(r)A) $$
for some diagonalisable linear operator $A\equiv \textnormal{diag}[\nu_1,\cdots,\nu_n]$ on $\mathfrak{g}.$
\item $\forall X,Y\in \mathfrak{g}, $ and $r>0,$ $[D_{r}^{\mathfrak{g}}X, D_{r}^{\mathfrak{g}}Y]=D_{r}^{\mathfrak{g}}[X,Y].$ 
\end{itemize}
We call  the eigenvalues of $A,$ $\nu_1,\nu_2,\cdots,\nu_n,$ the dilations weights or weights of $G$.  The homogeneous dimension of a homogeneous Lie group $G$ is given by  $$ Q=\textnormal{\textbf{Tr}}(A)=\nu_1+\cdots+\nu_n.  $$
The dilations $D_{r}^{\mathfrak{g}}$ of the Lie algebra $\mathfrak{g}$ induce a family of  maps on $G$ defined via,
$$ D_{r}:=\exp_{G}\circ D_{r}^{\mathfrak{g}} \circ \exp_{G}^{-1},\,\, r>0, $$
where $\exp_{G}:\mathfrak{g}\rightarrow G$ is the usual exponential mapping associated to the Lie group $G.$ We refer to the family $D_{r},$ $r>0,$ as dilations on the group. If we write $rx=D_{r}(x),$ $x\in G,$ $r>0,$ then a relation on the homogeneous structure of $G$ and the Haar measure $dx$ on $G$ is given by $$ \int\limits_{G}(f\circ D_{r})(x)dx=r^{-Q}\int\limits_{G}f(x)dx. $$
    
A  Lie group is graded if its Lie algebra $\mathfrak{g}$ may be decomposed as the sum of subspaces $\mathfrak{g}=\mathfrak{g}_{1}\oplus\mathfrak{g}_{2}\oplus \cdots \oplus \mathfrak{g}_{s}$ such that $[\mathfrak{g}_{i},\mathfrak{g}_{j} ]\subset \mathfrak{g}_{i+j},$ and $ \mathfrak{g}_{i+j}=\{0\}$ if $i+j>s.$  Examples of such groups are the Heisenberg group $\mathbb{H}^n$ and more generally any stratified groups where the Lie algebra $ \mathfrak{g}$ is generated by $\mathfrak{g}_{1}$.  Here, $n$ is the topological dimension of $G,$ $n=n_{1}+\cdots +n_{s},$ where $n_{k}=\mbox{dim}\mathfrak{g}_{k}.$

A Lie algebra admitting a family of dilations is nilpotent, and hence so is its associated
connected, simply connected Lie group. The converse does not hold, i.e., not every
nilpotent Lie group is homogeneous (see Dyer \cite{Dyer1970}) although they exhaust a large class (see Johnson \cite[page 294]{Johnson1975}). Indeed, the main class of Lie groups under our consideration is that of graded Lie groups. A graded Lie group $G$ is a homogeneous Lie group equipped with a family of weights $\nu_j,$ all of them positive rational numbers. Let us observe that if $\nu_{i}=\frac{a_i}{b_i}$ with $a_i,b_i$ integer numbers,  and $b$ is the least common multiple of the $b_i's,$ the family of dilations 
$$ \mathbb{D}_{r}^{\mathfrak{g}}=\textnormal{Exp}(\ln(r^b)A):\mathfrak{g}\rightarrow\mathfrak{g}, $$
have integer weights,  $\nu_{i}=\frac{a_i b}{b_i}. $ So, in this paper we always assume that the weights $\nu_j,$ defining the family of dilations are non-negative integer numbers which allow us to assume that the homogeneous dimension $Q$ is a non-negative integer number. This is a natural context for the study of Rockland operators (see Remark 4.1.4 of \cite{FischerRuzhanskyBook2015}).

\subsection{Fourier analysis on nilpotent Lie groups}

Let $G$ be a simply connected nilpotent Lie group.  
Let us assume that $\pi$ is a continuous, unitary and irreducible  representation of $G,$ this means that,
\begin{itemize}
    \item $\pi\in \textnormal{Hom}(G, \textnormal{U}(H_{\pi})),$ for some separable Hilbert space $H_\pi,$ i.e. $\pi(xy)=\pi(x)\pi(y)$ and for the  adjoint of $\pi(x),$ $\pi(x)^*=\pi(x^{-1}),$ for every $x,y\in G.$
    \item The map $(x,v)\mapsto \pi(x)v, $ from $G\times H_\pi$ into $H_\pi$ is continuous.
    \item For every $x\in G,$ and $W_\pi\subset H_\pi,$ if $\pi(x)W_{\pi}\subset W_{\pi},$ then $W_\pi=H_\pi$ or $W_\pi=\emptyset.$
\end{itemize} Let $\textnormal{Rep}(G)$ be the set of unitary, continuous and irreducible representations of $G.$ The relation, {\small{
\begin{equation*}
    \pi_1\sim \pi_2\textnormal{ if and only if, there exists } A\in \mathscr{B}(H_{\pi_1},H_{\pi_2}),\textnormal{ such that }A\pi_{1}(x)A^{-1}=\pi_2(x), 
\end{equation*}}}for every $x\in G,$ is an equivalence relation and the unitary dual of $G,$ denoted by $\widehat{G}$ is defined via
$
    \widehat{G}:={\textnormal{Rep}(G)}/{\sim}.
$ Let us denote by $d\pi$ the Plancherel measure on $\widehat{G}.$ 
The Fourier transform of $f\in \mathscr{S}(G), $ (this means that $f\circ \textnormal{exp}_G\in \mathscr{S}(\mathfrak{g})$, with $\mathfrak{g}\simeq \mathbb{R}^{\dim(G)}$) at $\pi\in\widehat{G},$ is defined by 
\begin{equation*}
    \widehat{f}(\pi)=\int\limits_{G}f(x)\pi(x)^*dx:H_\pi\rightarrow H_\pi,\textnormal{   and   }\mathscr{F}_{G}:\mathscr{S}(G)\rightarrow \mathscr{S}(\widehat{G}):=\mathscr{F}_{G}(\mathscr{S}(G)).
\end{equation*}

If we identify one representation $\pi$ with its equivalence class, $[\pi]=\{\pi':\pi\sim \pi'\}$,  for every $\pi\in \widehat{G}, $ the Kirillov trace character $\Theta_\pi$ defined by  $$(\Theta_{\pi},f):
=\textnormal{\textbf{Tr}}(\widehat{f}(\pi)),$$ is a tempered distribution on $\mathscr{S}(G).$ In particular, the identity
$
    f(e_G)=\int\limits_{\widehat{G}}(\Theta_{\pi},f)d\pi,
$ 
implies the Fourier inversion formula $f=\mathscr{F}_G^{-1}(\widehat{f}),$ where
\begin{equation*}
    (\mathscr{F}_G^{-1}\sigma)(x):=\int\limits_{\widehat{G}}\textnormal{\textbf{Tr}}(\pi(x)\sigma(\pi))d\pi,\,\,x\in G,\,\,\,\,\mathscr{F}_G^{-1}:\mathscr{S}(\widehat{G})\rightarrow\mathscr{S}(G),
\end{equation*}is the inverse Fourier  transform. In this context, the Plancherel theorem takes the form $\Vert f\Vert_{L^2(G)}=\Vert \widehat{f}\Vert_{L^2(\widehat{G})}$,  where  $$L^2(\widehat{G}):=\int\limits_{\widehat{G}}H_\pi\otimes H_{\pi}^*d\pi,$$ is the Hilbert space endowed with the norm: $\Vert \sigma\Vert_{L^2(\widehat{G})}=(\int_{\widehat{G}}\Vert \sigma(\pi)\Vert_{\textnormal{HS}}^2d\pi)^{\frac{1}{2}}.$
\subsection{The spaces $H^1$ and $BMO$ on homogeneous groups}
We will fix a homogeneous quasi-norm on $G,$ $|\cdot|.$ This means that $|\cdot|$ is a non-negative function on $G,$ satisfying 
\begin{equation*}
    |x|=|x^{-1}|,\,\,\,r|x|=|D_r( x)|,\,\,\,\textnormal{ and }|x|=0 \textnormal{ if and only if  }x=e_{G},
\end{equation*}
where $e_{G}$ is the identity element of $G.$ It satisfies a triangle
inequality with a constant: there exists a constant $\gamma\geq 1$ such that $|xy|\leq \gamma(|x|+|y|).$ As usual, the ball of radius $r>0,$ is defined as 
\begin{equation*}
    B(x,r)=\{y\in G:|y^{-1}x|<r\}.
\end{equation*}Then $BMO(G)$ is the space of locally integrable functions $f$ satisfying
\begin{equation*}
    \Vert f\Vert_{BMO(G)}:=\sup_{\mathbb{B}}\frac{1}{|\mathbb{B}|}\int\limits_{\mathbb{B}}|f(x)-f_{\mathbb{B}}|dx<\infty,\textnormal{ where  } f_{\mathbb{B}}:=\frac{1}{|\mathbb{B}|}\int\limits_{\mathbb{B}}f(x)dx,
\end{equation*}
and $\mathbb{B}$ ranges over all balls $B(x_{0},r),$ with $(x_0,r)\in G\times (0,\infty).$ The Hardy space $H^1(G)$ will be defined via the atomic decomposition. Indeed, $f\in H^1(G),$ if and only if, $f$ can be expressed as $f=\sum_{j=1}^\infty c_{j}a_{j},$ where $\{c_j\}_{j=1}^\infty$ is a sequence in $\ell^1(\mathbb{N}_0),$ and every function $a_j$ is an atom, i.e., $a_j$ is supported in some ball $B_j,$ $\int_{B_j}a_{j}(x)dx=0,$ and 
\begin{equation*}
    \Vert a_j\Vert_{L^\infty(G)}\leq \frac{1}{|B_j|}.
\end{equation*} The norm $\Vert f\Vert_{H^1(G)}$ is the infimum over  all possible series $\sum_{j=1}^\infty|c_j|.$ Furthermore $BMO(G)$ is the dual of $H^1(G),$ (see Folland and Stein \cite{FollandStein1982}). This can be understood in the following sense:
\begin{itemize}
    \item[(a).] If $\phi\in BMO(G), $ then $\Phi: f\mapsto \int\limits_{G}f(x)\phi(x)dx,$ admits a bounded extension on $H^1(G).$
    \item[(b).] Conversely, every continuous linear functional $\Phi$ on $H^1(G)$ arises as in $\textnormal{(a)}$ with a unique element $\phi\in BMO(G).$
\end{itemize} The norm of $\phi$ as a linear functional on $H^1(G)$ is equivalent with the $BMO(G)$-norm. Important properties of the $BMO(G)$ and the $H^1(G)$ norms are the following,
\begin{equation}\label{BMOnormduality}
 \Vert f \Vert_{BMO(G)}  =\sup_{\Vert g\Vert_{H^1(G)}=1} 
\left| \int\limits_{G}f(x)g(x)dx\right|,\end{equation}
\begin{equation}\label{BMOnormduality'}\Vert g \Vert_{H^1(G)}  =\sup_{\Vert f\Vert_{BMO(G)}=1} 
\left| \int\limits_{G}f(x)g(x)dx\right|.
\end{equation} 
The identities  \eqref{BMOnormduality} and \eqref{BMOnormduality'}, will be important in the duality argument at the end of the proof of Theorem \ref{LinftyBMOCardonaDelgadoRuzhansky}.

\subsection{Homogeneous linear operators and Rockland operators} A linear operator $T:C^\infty(G)\rightarrow \mathscr{D}'(G)$ is homogeneous of  degree $\nu\in \mathbb{C}$ if for every $r>0$ the equality 
\begin{equation*}
T(f\circ D_{r})=r^{\nu}(Tf)\circ D_{r}
\end{equation*}
holds for every $f\in \mathscr{D}(G). $
If for every representation $\pi\in\widehat{G},$ $\pi:G\rightarrow U({H}_{\pi}),$ we denote by ${H}_{\pi}^{\infty}$ the set of smooth vectors, that is, the space of elements $v\in {H}_{\pi}$ such that the function $x\mapsto \pi(x)v,$ $x\in \widehat{G}$ is smooth,  a Rockland operator is a left-invariant differential operator $\mathcal{R}$ which is homogeneous of positive degree $\nu=\nu_{\mathcal{R}}$ and such that, for every unitary irreducible non-trivial representation $\pi\in \widehat{G},$ $\pi({R})$ is injective on ${H}_{\pi}^{\infty};$ $\sigma_{\mathcal{R}}(\pi)=\pi(\mathcal{R})$ is the symbol associated to $\mathcal{R}.$ It coincides with the infinitesimal representation of $\mathcal{R}$ as an element of the universal enveloping algebra. It can be shown that a Lie group $G$ is graded if and only if there exists a differential Rockland operator on $G.$ If the Rockland operator is formally self-adjoint, then $\mathcal{R}$ and $\pi(\mathcal{R})$ admit self-adjoint extensions on $L^{2}(G)$ and ${H}_{\pi},$ respectively. Now if we preserve the same notation for their self-adjoint
extensions and we denote by $E$ and $E_{\pi}$  their spectral measures, we will denote by
$$ f(\mathcal{R})=\int\limits_{-\infty}^{\infty}f(\lambda) dE(\lambda),\,\,\,\textnormal{and}\,\,\,\pi(f(\mathcal{R}))\equiv f(\pi(\mathcal{R}))=\int\limits_{-\infty}^{\infty}f(\lambda) dE_{\pi}(\lambda), $$ the functions defined by the functional calculus. 
In general, we will reserve the notation ${dE_A(\lambda)}_{0\leq\lambda<\infty}$ for the spectral measure associated with a positive and self-adjoint operator $A$ on a Hilbert space $H.$ 

We now recall a lemma on dilations on the unitary dual $\widehat{G},$ which will be useful in our analysis of spectral multipliers.   For the proof, see Lemma 4.3 of \cite{FischerRuzhanskyBook2015}.
\begin{lemma}\label{dilationsrepre}
For every $\pi\in \widehat{G}$ let us define  $D_{r}(\pi)(x)=\pi(rx)$ for every $r>0$ and $x\in G.$ Then, if $f\in L^{\infty}(\mathbb{R})$ then $f(\pi^{(r)}(\mathcal{R}))=f({r^{\nu}\pi(\mathcal{R})}).$
\end{lemma}
 {{For instance, for any $\alpha\in \mathbb{N}_0^n,$ and for an arbitrary family $X_1,\cdots, X_n,$ of left-invariant  vector-fields we will use the notation
\begin{equation}
    [\alpha]:=\sum_{j=1}^n\nu_j\alpha_j,
\end{equation}for the homogeneity degree of the operator $X^{\alpha}:=X_1^{\alpha_1}\cdots X_{n}^{\alpha_n},$ whose order is $|\alpha|:=\sum_{j=1}^n\alpha_j.$}}

\subsection{Symbols and quantization of pseudo-differential operators}
 In order to present a consistent definition of pseudo-differential operators one developed in \cite{FischerRuzhanskyBook2015} (see the quantisation formula  \eqref{Quantization}),  a suitable class of spaces on the unitary dual $\widehat{G}$ acting in a suitable way with the set of smooth vectors $H_{\pi}^{\infty},$ on every representation space $H_{\pi}.$ Let now recall the main notions.
 \begin{definition}[Sobolev spaces on  smooth vectors] Let $\pi_1\in \textnormal{Rep}(G),$ and $a\in \mathbb{R}.$ We denote by $H_{\pi_1}^a,$ the Hilbert space obtained  as the completion of $H_{\pi_1}^\infty$ with respect to the norm
 \begin{equation*}
     \Vert v \Vert_{H_{\pi_1}^a}=\Vert\pi_1(1+\mathcal{R})^{\frac{a}{\nu}} v\Vert_{H_{\pi_1}},
 \end{equation*}
 where $\mathcal{R}$ is  a positive Rockland operator on $G$ of homogeneous degree $\nu>0.$
 \end{definition}
 
 In order to introduce the general notion of symbol as the one developed in \cite{FischerRuzhanskyBook2015}, we will use a suitable notion of operator-valued  symbols acting on smooth vectors. We introduce it as follows.
 
\begin{definition}
A $\widehat{G}$-field of operators $\sigma=\{\sigma(\pi):\pi\in \widehat{G}\}$ defined on smooth vectors {is defined} on the Sobolev space ${H}_\pi^a$ when for each representation $\pi_1\in \textnormal{Rep}(G),$ the operator $\sigma(\pi_1)$ is bounded from $H^a_{\pi_1}$ into $H_{\pi_{1}}$ in the sense that
\begin{equation*}
    \sup_{ \Vert v\Vert_{H_{\pi_1}^a}=1}\Vert \sigma(\pi_1)v \Vert<\infty.
\end{equation*}
\end{definition}

We will consider those $\widehat{G}-$fields of operators with ranges in Sobolev spaces on smooth vectors. This is useful, for example,  in the generalization of the Calder\'on-Vaillancourt theorem and in establishing that the class of pseudo-differential operators introduced in Section \ref{sec3}  have Calder\'on-Zygmund kernels (see \cite{FischerRuzhanskyBook2015} for instance). We recall that the  Sobolev space $L^{2}_{a}(G)$ is  defined by the norm (see \cite[Chapter 4]{FischerRuzhanskyBook2015})
\begin{equation}\label{L2ab2}
    \Vert f \Vert_{L^{2}_{a}(G)}=\Vert (1+\mathcal{R})^{\frac{a}{\nu}}f\Vert_{L^2(G)},
\end{equation} for $s\in \mathbb{R}.$

\begin{definition}
A $\widehat{G}$-field of operators  defined on smooth vectors {with range} in the Sobolev space $H_{\pi}^a$ is a family of classes of operators $\sigma=\{\sigma(\pi):\pi\in \widehat{G}\}$ where
\begin{equation*}
    \sigma(\pi):=\{\sigma(\pi_1):H^{\infty}_{\pi_1}\rightarrow H_{\pi}^a,\,\,\pi_1\in \pi\},
\end{equation*} for every $\pi\in \widehat{G}$ viewed as a subset of $\textnormal{Rep}(G),$ satisfying for every two elements $\sigma(\pi_1)$ and $\sigma(\pi_2)$ in $\sigma(\pi):$
\begin{equation*}
  \textnormal{If  }  \pi_{1}\sim \pi_2  \textnormal{  then  }   \sigma(\pi_1)\sim \sigma(\pi_2). 
\end{equation*}
\end{definition}
  The following notion  will be useful in order to use the general theory of non-commutative integration (see e.g. Dixmier \cite{Dixmier1953}). 
\begin{definition}
A $\widehat{G}$-field of operators  defined on smooth vectors with range in the Sobolev space $H_\pi^a$ {is measurable}  when for some (and hence for any) $\pi_1\in \pi$ and any vector $v_{\pi_1}\in H_{\pi_1}^\infty,$ as $\pi\in \widehat{G},$ the resulting field $\{\sigma(\pi)v_\pi:\pi\in\widehat{G}\},$ 
is $d\pi$-measurable and
\begin{equation*}
    \int\limits_{\widehat{G} }\Vert v_\pi \Vert^2_{H_\pi^a}d\pi=\int\limits_{\widehat{G} }\Vert\pi(1+\mathcal{R})^{\frac{a}{\nu}} v_\pi \Vert^2_{H_\pi}d\pi<\infty.
\end{equation*}
\end{definition}
\begin{remark}
We always assume that a $\widehat{G}$-field of operators  defined on smooth vectors {with range} in the Sobolev space $H_{\pi}^a$ is $d\pi$-measurable.
\end{remark} The $\widehat{G}$-fields of operators associated to Rockland operators can be defined as follows.
\begin{definition}
Let $L^2_a(\widehat{G})$ denote the space of fields of operators $\sigma$ with range in $H_\pi^a,$ that is,
\begin{equation*}
    \sigma=\{\sigma(\pi):H_\pi^\infty\rightarrow H_\pi^a\}, \textnormal{ with }\{\pi(1+\mathcal{R})^{\frac{a}{\nu}}\sigma(\pi):\pi\in \widehat{G}\}\in L^2(\widehat{G}),
\end{equation*}for one (and hence for any) Rockland operator of homogeneous degree $\nu.$ We also denote
\begin{equation*}
    \Vert \sigma\Vert_{L^2_a(\widehat{G})}:=\Vert \pi(1+\mathcal{R})^{\frac{a}{\nu}}\sigma(\pi)\Vert_{L^2(\widehat{G})}.
\end{equation*}
\end{definition} By using the notation above, we will introduce a family of function spaces required to define $\widehat{G}$-fields of operators (that will be used to define the symbol of a pseudo-differential operator, see Definition \ref{SRCK}).
\begin{definition}[The spaces $\mathscr{L}_{L}(L^2_a(G),L^2_b(G)),$ $\mathcal{K}_{a,b}(G)$ and $L^\infty_{a,b}(\widehat{G})$]
\hspace{0.1cm}
\begin{itemize}
    \item The space $\mathscr{L}_{L}(L^2_a(G),L^2_b(G)) $ consists of all  left-invariant operators $T$  such that  $T:L^2_a(G)\rightarrow L^2_b(G) $ extends to a bounded operator.
    \item The space  $\mathcal{K}_{a,b}(G)$ is the family of all right convolution kernels of elements in $  \mathscr{L}_{L}(L^2_a(G),L^2_b(G))  ,$ i.e. $k=T\delta\in \mathcal{K}_{a,b}(G)$ if and only if  $T\in
    \mathscr{L}_{L}(L^2_a(G),L^2_b(G)) .$ 
    \item We also define the space $L^\infty_{a,b}(\widehat{G})$ by the following condition:  $\sigma\in L^\infty_{a,b}(\widehat{G})$ if 
\begin{equation*}
    \Vert \pi(1+\mathcal{R})^{\frac{b}{\nu}}\sigma(\pi)\pi(1+\mathcal{R})^{-\frac{a}{\nu}} \Vert_{L^\infty(\widehat{G})}:=\sup_{\pi\in\widehat{G}}\Vert \pi(1+\mathcal{R})^{\frac{b}{\nu}}\sigma(\pi)\pi(1+\mathcal{R})^{-\frac{a}{\nu}} \Vert_{\mathscr{B}(H_\pi)}<\infty.
\end{equation*}
\end{itemize} In this case $T_\sigma:L^2_a(G)\rightarrow L^2_b(G)$ extends to a bounded operator with 
\begin{equation*}
  \Vert \sigma\Vert _{L^\infty_{a,b}(\widehat{G})}= \Vert T_{\sigma} \Vert_{\mathscr{L}(L^2_a(G),L^2_b(G))},
\end{equation*} and   $\sigma\in L^\infty_{a,b}(\widehat{G})$ if and only if $k:=\mathscr{F}_{G}^{-1}\sigma \in \mathcal{K}_{a,b}(G).$
\end{definition}
   With the previous definitions, we will introduce the type of symbols that we will use  in this work and under which the quantization formula makes sense.
   
   \begin{definition}[Symbols and right-convolution kernels]\label{SRCK} A {symbol} is a field of operators $\{\sigma(x,\pi):H_\pi^\infty\rightarrow H_\pi,\,\,\pi\in\widehat{G}\},$ depending on $x\in G,$ such that 
   \begin{equation*}
       \sigma(x,\cdot)=\{\sigma(x,\pi):H_\pi^\infty\rightarrow H_\pi,\,\,\pi\in\widehat{G}\}\in L^\infty_{a,b}(\widehat{G})
   \end{equation*}for some $a,b\in \mathbb{R}.$ The {right-convolution kernel} $k\in C^\infty(G,\mathscr{S}'(G))$ associated with $\sigma$ is defined, via the inverse Fourier transform on the group by
   \begin{equation*}
       x\mapsto k(x)\equiv k_{x}:=\mathscr{F}_{G}^{-1}(\sigma(x,\cdot)): G\rightarrow\mathscr{S}'(G).
   \end{equation*}
   \end{definition}
   Definition \ref{SRCK} in this section allows us to establish the following theorem, which gives sense to the quantization of pseudo-differential operators in the graded setting (see Theorem 5.1.39 of \cite{FischerRuzhanskyBook2015}).
   \begin{theorem}\label{thetheoremofsymbol}
   Let us consider a symbol $\sigma$ and its associated right-convolution  kernel $k.$ For every $f\in \mathscr{S}(G),$ let us define the operator $A$ acting on $\mathscr{S}(G),$ via 
   \begin{equation}\label{pseudo}
       Af(x)=(f\ast k_{x})(x),\,\,\,\,\,x\in G.
   \end{equation}Then $Af\in C^\infty,$ and 
   \begin{equation}\label{Quantization}
       Af(x)=\int\limits_{\widehat{G}}\textnormal{\textbf{Tr}}(\pi(x)\sigma(x,\pi)\widehat{f}(\pi))d\pi.
   \end{equation}
   \end{theorem}
   
 Theorem   \ref{thetheoremofsymbol} motivates the following definition.
 \begin{definition}\label{DefiPSDO}
A continuous linear operator $A:C^\infty(G)\rightarrow\mathscr{D}'(G)$ with Schwartz kernel $K_A\in C^{\infty}(G)\widehat{\otimes}_{\pi} \mathscr{D}'(G),$ is a pseudo-differential operator, if
 there exists a  \textit{symbol}, which is a field of operators $\{\sigma(x,\pi):H_\pi^\infty\rightarrow H_\pi,\,\,\pi\in\widehat{G}\},$ depending on $x\in G,$ such that 
   \begin{equation*}
       \sigma(x,\cdot)=\{\sigma(x,\pi):H_\pi^\infty\rightarrow H_\pi,\,\,\pi\in\widehat{G}\}\in L^\infty_{a,b}(\widehat{G})
   \end{equation*}for some $a,b\in \mathbb{R},$ such that, the Schwartz kernel of $A$ is given by 
\begin{equation*}
    K_{A}(x,y)=\int\limits_{\widehat{G}}\textnormal{\textbf{Tr}}(\pi(y^{-1}x)\sigma(x,\pi))d\pi=k_{x}(y^{-1}x).
\end{equation*}
{{In this case, we use  the notation
\begin{equation*}
    A:=\textnormal{Op}(\sigma),
\end{equation*}     
to indicate that $A$ is the pseudo-differential operator associated with  $\sigma.$}}
\end{definition}

  \begin{remark}[Hilbert-Schmidt operators]
  Let $A$ in \eqref{pseudo} be a pseudo-differential operator with symbol $\sigma.$ The operator $A:L^2({G})\rightarrow L^2({G})$ extends to a Hilbert-Schmidt operator, if and only if, $\int\int_{G\times G}|K_A(x,y)|^2dxdy<\infty.$ Equivalently, the Plancherel theorem, allows us to express the last condition  as
  \begin{align*}
    &\int\limits_{G}\int\limits_{G}|K_A(x,y)|^2dydx =\int\limits_{G}\int\limits_{G}|k_x(y^{-1}x)|^2dydx=\int\limits_{G}\int\limits_{G}|k_x(z)|^2dzdx\\
    &=\int\limits_{G}\int\limits_{\widehat{G}}\Vert\sigma(x,\pi) \Vert_{\textnormal{HS}}^2d\pi dx=\Vert \sigma(\cdot,\cdot) \Vert^2_{L^2(G\times \widehat{G})}<\infty.
  \end{align*}So, in terms of the symbol $\sigma,$ $A:L^2({G})\rightarrow L^2({G})$ extends to a Hilbert-Schmidt operator if and only if $\sigma\in L^2(G\times \widehat{G}).$
  \end{remark} 
   \begin{remark}[$L^\infty(G)$-boundedness of pseudo-differential operators]\label{Linftyremark} If we assume the following condition
   \begin{equation*}
       \sup_{x\in G}\Vert\mathscr{F}_{G}^{-1}\sigma(x,\cdot) \Vert_{L^1(G)}=\sup_{x\in G}\Vert k_{x} \Vert_{L^1(G)}<\infty,
   \end{equation*} then 
   \begin{equation*}
       |Af(x)|\leq \int\limits_{G}|k_x(y^{-1}x)||f(y)|dy\leq  \sup_{x\in G}\Vert k_{x} \Vert_{L^1(G)}\Vert f\Vert_{L^\infty(G)}.
   \end{equation*}Consequently,
   $$\Vert A\Vert_{\mathscr{B}(L^\infty(G))}\leq \sup_{x\in G}\Vert k_{x} \Vert_{L^1(G)}. $$ This simple fact will be used in our further analysis.
   \end{remark}
\section{Computing the symbol of a continuous linear operator}
Let $\mathcal{R}$ be a positive Rockland operator on a graded Lie group. Then $\mathcal{R}$  and $\pi(\mathcal{R}):=d\pi(\mathcal{R})$ (the infinitesimal representation of $\mathcal{R}$) are symmetric and  densely defined operators on $C^\infty_0 (G)$ and $H^\infty_\pi\subset H_\pi.$ We will denote by $\mathcal{R}$  and $\pi(\mathcal{R}):=d\pi(\mathcal{R})$ their self-adjoint extensions to $L^2(G)$ and $H_\pi$ respectively (see Proposition 4.1.5 and Corollary 4.1.16 of \cite[page 178]{FischerRuzhanskyBook2015}). 

\begin{remark}\label{symbolremark}
Let $\mathcal{R}$ be a positive  Rockland operator of  homogeneous degree $\nu$ on a graded Lie group $G.$ Every operator $  \pi(\mathcal{R}) $ has discrete spectrum (see ter Elst and Robinson \cite{TElst+Robinson}) admitting, by the spectral theorem, a basis  contained in its domain. In this case, $H_{\pi}^{\infty}\subset\textnormal{Dom}(\pi(\mathcal{R}))\subset H_{\pi},$ but in view of Proposition 4.1.5 and Corollary 4.1.16 of \cite[page 178]{FischerRuzhanskyBook2015}, every $ \pi(\mathcal{R})$ is densely defined and symmetric on $H^\infty_\pi,$ and this fact allows us to define the (restricted) domain of $\pi(\mathcal{R}),$ as \begin{equation}\label{RestrictedDomain}
    \textnormal{Dom}_{\textnormal{rest}}(\pi(\mathcal{R}))=H_{\pi}^{\infty}.   
\end{equation}Next, when we mention the domain of $\pi(\mathcal{R})$ we are referring to the restricted domain in  \eqref{RestrictedDomain}. This fact will be important, because, via the spectral theorem we can construct a basis for $H_{\pi},$ consisting of vectors in $\textnormal{Dom}_{\textnormal{rest}}(\pi(\mathcal{R}))=H_{\pi}^{\infty},$ where the operator $\pi(\mathcal{R})$ is diagonal. So, if  $B_\pi=\{e_{\pi,k}\}_{k=1}^\infty\subset H_{\pi}^{\infty},$ is  a basis such that $\pi(\mathcal{R})$  satisfies

$ \pi(\mathcal{R})e_{\pi,k}=\lambda_{\pi,k}e_{\pi,k},\,k\in \mathbb{N}, \,\,\,\pi \in \widehat{G}, $ for every $x\in G,$ the function $x\mapsto \pi(x)e_{\pi,k},$ is smooth and the family of functions
\begin{equation}\label{piij}
   \pi_{ij}:G\rightarrow \mathbb{C},\,\, \pi(x)_{ij}:=( \pi(x)e_{\pi,i},e_{\pi,j} )_{H_\pi},\,\,x\in G,
\end{equation} are smooth functions on $G.$  Consequently, for every continuous linear operator $A:C^\infty(G)\rightarrow C^\infty(G)$ we have \begin{equation*}\{\pi_{ij}\}_{i,j=1}^\infty\subset \textnormal{Dom}(A)=C^\infty(G),\end{equation*} for every $ \pi\in \widehat{G}. $
\end{remark} 

In view of Remark \ref{symbolremark}  we have the following theorem where we present the formula of a global symbol in terms of its corresponding pseudo-differential operator in the graded setting.

\begin{theorem}\label{SymbolPseudo}
Let  $\mathcal{R}$ be a positive  Rockland operator of  homogeneous degree $\nu$ on a graded Lie group $G.$ For every $\pi\in \widehat{G},$ let  $B_\pi=\{e_{\pi,k}\}_{k=1}^\infty\subset H_{\pi}^{\infty},$ be  a basis where the operator $\pi(\mathcal{R})$ is diagonal, i.e., 
\begin{equation*}
    \pi(\mathcal{R})e_{\pi,k}=\lambda_{\pi,k}e_{\pi,k},\,k\in \mathbb{N}, \,\,\,\pi \in \widehat{G}.
\end{equation*} For every $x\in G,$ and $\pi\in \widehat{G},$ let us consider the functions $\pi(\cdot)_{ij}\in C^\infty(G)$ in \eqref{piij} induced by the coefficients of the matrix representation
 of $\pi(x)$ in the basis $B_\pi.$ If $A:C^\infty(G)\rightarrow C^\infty(G)$ is a continuous linear operator with symbol 
\begin{equation}\label{definition}
    \sigma:=\{\sigma(x,\pi)\in \mathscr{L}(H_{\pi}^\infty, H_\pi):\, x\in G,\,\pi\in \widehat{G}\},\, 
\end{equation} such that
\begin{equation}\label{quantization}
    Af(x)=\int\limits_{\widehat{G}}\textnormal{\textbf{Tr}}(\pi(x)\sigma(x,\pi)\widehat{f}(\pi))d\pi,
\end{equation} for every $f\in \mathscr{S}(G),$ and a.e. $(x,\pi),$ and if   $A\pi(x)$ is the  densely defined operator on $H_{\pi}^\infty,$ via
\begin{equation}\label{thesymbolofA}
  A\pi(x)\equiv  ((A\pi(x)e_{\pi,i},e_{\pi,j}))_{i,j=1}^\infty,\,\,\,(A\pi(x)e_{\pi,i},e_{\pi,j})=:(A\pi_{ij})(x),
\end{equation} then we have
\begin{equation}\label{formulasymbol}
    \sigma(x,\pi)=\pi(x)^*A\pi(x),
\end{equation} for every $x\in G,$ and a.e. $\pi\in \widehat{G}.$ 
\end{theorem}

\begin{proof}
Let $x\in G$ and $f\in \mathscr{S}(G).$ The Fourier inversion formula gives
\begin{equation*}
    f(x)=\int\limits_{\widehat{G}}\textnormal{\textbf{Tr}}(\pi(x)\widehat{f}(\pi))d\pi.
\end{equation*}  We will show that $\sigma$ defined by \eqref{formulasymbol} via \eqref{thesymbolofA}, satisfies \eqref{quantization}. Once we do this, the symbol $\sigma$ satisfying \eqref{quantization} is unique. Indeed, if there exists $A',$ defined by
\begin{equation}\label{quantization'}
    A' f(x)=\int\limits_{\widehat{G}}\textnormal{\textbf{Tr}}(\widehat{f}(\pi){\sigma}'(x,\pi)\pi(x))d\pi,
\end{equation} for every $f\in C^\infty_{0}(G),$ such that  $A=A' $ on $C^\infty(G),$  then $\sigma(x,\pi)=\sigma'(x,\pi)$ for $a.e$ $x\in G,$ and $\pi\in \widehat{G}.$ The equality of this two operators is understood in the sense that  $\sigma(x,\pi)v=\sigma'(x,\pi)v$ for every $v\in H^\infty_\pi.$ Because $A,A':C^\infty(G)\rightarrow C^\infty(G)$ are continuous linear operators, both can be extended as continuous linear operator on the whole space of distributions, and we will denote by $A,A': \mathscr{D}'(G)\rightarrow \mathscr{D}'(G)$  these extensions. Clearly, $A=A'$ on $ \mathscr{D}'(G).$ In particular, if $\delta_{e_G}$ is the Delta distribution at the identity element of $G,$ we have
\begin{equation*}
    k_{x}(\cdot)=(A\delta_{e_G})(x)=(A'\delta_{e_G})(x)=k_{x}'(\cdot)\in \mathscr{D}'(G),\,\,x\in G,
\end{equation*} were $x\mapsto k_{x}(\cdot)$ and  $x\mapsto k_{x}'(\cdot)$ are the   convolution kernels of $A$ and $A'$ respectively. In view of the identities
\begin{equation*}
    \sigma(x,\pi)=\widehat{k}_x(\pi)=\widehat{k}_x'(\pi)=\sigma'(x,\pi),
\end{equation*} we conclude the uniqueness of $\sigma$. Now, assume that $\sigma$ is defined by \eqref{formulasymbol} via \eqref{thesymbolofA}. In order to prove \eqref{quantization}, we compute explicitly $A\pi_{ij}(x).$ Indeed,
\begin{align*}
    A\pi_{ij}(x)&=\int\limits_{G}k_{x}(y^{-1}x)\pi_{ij}(y)dy\\&=\int\limits_{G}k_{x}(y^{-1}x)(\pi(y)e_{\pi,i},e_{\pi,j})_{H_\pi}dy\\
    &=\left(\int\limits_{G}k_{x}(y^{-1}x)\pi(y)dye_{\pi,i},e_{\pi,j}\right)_{H_\pi}.
\end{align*}
The variable change $z:=y^{-1}x,$ allows us to write
\begin{align*}
 A\pi_{ij}(x)&=\left(\int\limits_{G}k_{x}(y^{-1}x)\pi(y)dye_{\pi,i},e_{\pi,j}\right)_{H_\pi}\\
 &=\left(\int\limits_{G}k_{x}(z)\pi(xz^{-1})dze_{\pi,i},e_{\pi,j}\right)_{H_\pi}\\&= \left(\pi(x)\int\limits_{G}k_{x}(z)\pi(z)^{*}dze_{\pi,i},e_{\pi,j}\right)_{H_\pi}\\
 &= \left(\pi(x)\sigma(x,\pi)e_{\pi,i},e_{\pi,j}\right)_{H_\pi}.
\end{align*}
So, we have proved that in terms of the basis $B_\pi,$ the  $(i,j)$-entry  in the matrix-representation of $A\pi(x),$ agrees with the  $(i,j)$-element  in the matrix-representation of the operator  $\pi(x)\sigma(x,\pi).$ Consequently, we have the equality of operators,
\begin{equation*}
    A\pi(x)=\pi(x)\sigma(x,\pi):H_{\pi}^\infty\rightarrow H_\pi.
\end{equation*}
By using that $\pi(x):H_{\pi}\rightarrow H_\pi$ is a unitary operator, and that $\pi(x)^{*}=\pi(x)^{-1},$ we conclude that
\begin{equation*}
    \sigma(x,\pi)=\pi(x)^*A\pi(x).
\end{equation*}
Thus,  we finish the proof.
\end{proof}
\begin{example}
To illustrate Theorem \ref{SymbolPseudo},  let us consider $A=(\varepsilon+\mathcal{R})^{\frac{a}{\nu}}$ where $a\in \mathbb{R}$ and $\varepsilon>0$ (or $a>0$ and $\varepsilon\geq 0$). Because $A$ is in this case a left-invariant operator, the symbol of $A,$ $\sigma_{A}(\pi)=\sigma(x,\pi)=(\varepsilon+\pi(\mathcal{R}))^{\frac{a}{\nu}}$ is independent of $x\in G.$ Thus, Theorem \ref{SymbolPseudo} implies that
\begin{equation*}
    (\varepsilon+\pi(\mathcal{R}))^{\frac{a}{\nu}}=\pi(x)^*(\varepsilon+\mathcal{R})^{\frac{a}{\nu}}\pi(x).
\end{equation*}In particular, for $x=e_{G}$ we have
\begin{equation*}
    (\varepsilon+\pi(\mathcal{R}))^{\frac{a}{\nu}}=(\varepsilon+\mathcal{R})^{\frac{a}{\nu}}\pi(e_G),\,\,\,\,\,\,\hspace{0.4cm}
\end{equation*}where $e_G$ is the identity element of $G.$ In particular, for $a=\nu$ and $\varepsilon=0,$ we have the identity $\pi(\mathcal{R})=(\mathcal{R}\pi)(e_G).$ 
\end{example}
\begin{remark}
If $A:C^\infty(G)\rightarrow C^\infty(G),$ is a (left-invariant operator) Fourier multiplier, from \eqref{formulasymbol}, the global symbol $\sigma$ of $A$ can be computed by the identity $\sigma(\pi)=\pi(x)^*A\pi(x).$ In particular, for $x=e_G,$ we have $\sigma(\pi)=A\pi(e_{G}),$ for a.e. $\pi\in\widehat{G}.$ Properties for Fourier multipliers on $L^{p}(G)$-spaces including the H\"ormander-Mihlin condition in terms of a difference structure on the unitary dual  can be found in \cite{FischerRuzhansky2014}.
\end{remark}

\section{Estimates for pseudo-differential operators}\label{sec3} 
In this section we prove our main theorem. We will use the global version of the Calder\'on-Vaillancourt theorem stated in \cite{FischerRuzhanskyBook2015} in order to prove the $L^\infty$-$BMO$ boundedness for certain   classes of pseudo-differential operators. Later we will conclude our analysis by using the Fefferman-Stein complex interpolation.  

\subsection{Global H\"ormander classes of pseudo-differential operators}

The main tool in the construction of global H\"ormander classes is the notion of difference operators. 
 Indeed, for every smooth function $q\in C^\infty(G)$ and $\sigma\in L^\infty_{a,b}(G),$ where $a,b\in\mathbb{R},$ the difference operator $\Delta_q$ acts on $\sigma$ according to the formula (see Definition 5.2.1 of \cite{FischerRuzhanskyBook2015}),
 \begin{equation*}
     \Delta_q\sigma(\pi)\equiv [\Delta_q\sigma](\pi):=\mathscr{F}_{G}(qf)(\pi),\,\textnormal{ for  a.e.  }\pi\in\widehat{G},\textnormal{ where }f:=\mathscr{F}_G^{-1}\sigma\,\,.
 \end{equation*}
We will reserve the notation $\Delta^{\alpha}$ for the difference operators defined by the functions $q_{\alpha}$ and $\tilde{q}_{\alpha}$ defined by $q_{\alpha}(x):=x^\alpha$ and $\tilde{q}_{\alpha}(x)=(x^{-1})^\alpha,$ respectively. In particular, we have the Leibnitz rule,
\begin{equation}\label{differenceespcia;l}
    \Delta^{\alpha}(\sigma\tau)=\sum_{\alpha_1+\alpha_2=\alpha}c_{\alpha_1,\alpha_2}\Delta^{\alpha_1}(\sigma)\Delta^{\alpha_2}(\tau),\,\,\,\sigma,\tau\in L^\infty_{a,b}(\widehat{G}).
\end{equation} 
For our further analysis we will use the following property of the difference operators $\Delta^\alpha,$ (see e.g. \cite[page 20]{FischerFermanian-Kammerer2017}),
\begin{equation}\label{FischerFermanian-Kammerer2017}
    \Delta^{\alpha}(\sigma_{r\cdot})(\pi)=r^{ [\alpha] }(\Delta^{\alpha}\sigma)(r\cdot \pi),\,\,\,r>0\,\,\,\pi\in\widehat{G},
\end{equation}where we have denoted \begin{equation}\label{eeeeeeeee}
    \sigma_{r\cdot}:=\{\sigma(r\cdot \pi):\pi\in\widehat{G}\},\,\,r\cdot \pi(x):=\pi(D_r(x)),\,\,x\in G. 
\end{equation}

In terms of difference operators, the global H\"ormander classes introduced in \cite{FischerRuzhanskyBook2015} can be introduced as follows.
 Let $0\leq \delta,\rho\leq 1,$ and let $\mathcal{R}$ be a positive Rockland operator of homogeneous degree $\nu>0.$ If $m\in \mathbb{R},$ we say that the symbol $\sigma\in L^\infty_{a,b}(\widehat{G}), $ where $a,b\in\mathbb{R},$ belongs to the $(\rho,\delta)$-H\"ormander class of order $m,$ $S^m_{\rho,\delta}(G\times \widehat{G}),$ if for all $\gamma\in \mathbb{R},$ the following conditions
\begin{equation}\label{seminorm}
   p_{\alpha,\beta,\gamma,m}(\sigma)= \operatornamewithlimits{ess\, sup}_{(x,\pi)\in G\times \widehat{G}}\Vert \pi(1+\mathcal{R})^{\frac{\rho [\alpha] -\delta [\beta] -m-\gamma}{\nu}}[X_{x}^\beta \Delta^{\alpha}\sigma(x,\pi)] \pi(1+\mathcal{R})^{\frac{\gamma}{\nu}}\Vert_{\textnormal{op}}<\infty,
\end{equation}
hold true for all $\alpha$ and $\beta$ in $\mathbb{N}_0^n.$ The resulting class $S^m_{\rho,\delta}(G\times \widehat{G}),$ does not depend on the choice of the Rockland operator $\mathcal{R}.$ In particular (see Theorem 5.5.20 of \cite{FischerRuzhanskyBook2015}), the following facts are equivalents: 
\begin{itemize}
    \item $\forall \alpha,\beta\in \mathbb{N}_{0}^n, \forall\gamma\in \mathbb{R}, $   $p_{\alpha,\beta,\gamma,m}(\sigma)<\infty.$
    
    \item  $\forall \alpha,\beta\in \mathbb{N}_{0}^n, $   $p_{\alpha,\beta,0,m}(\sigma)<\infty.$
    
    \item  $\forall \alpha,\beta\in \mathbb{N}_{0}^n, $   $p_{\alpha,\beta,m+\delta [\beta] -\rho [\alpha] ,m}(\sigma)<\infty.$
\end{itemize}
We will denote,
\begin{equation}
     \Vert \sigma\Vert_{k\,,S^{m}_{\rho,\delta}}= \max_{ [\alpha] + [\beta] \leq k}\{  p_{\alpha,\beta,0,m}(\sigma)\}.
\end{equation}
\begin{remark}
In the abelian case $G=\mathbb{R}^n,$ endowed in its natural structure of group, and $\mathcal{R}=-\Delta_{x},$ $x\in \mathbb{R}^n,$ with $\Delta_{x}=\sum_{j=1}^{n}\partial_{x_i}^{2}$ being the usual Laplace operator on $\mathbb{R}^n,$ the classes defined via \eqref{seminorm}, agree with the well known H\"ormander classes on $\mathbb{R}^n$ (see e.g. H\"ormander \cite[Vol. 3]{HormanderBook34}). In this case the difference operators are the partial derivatives on $\mathbb{R}^n,$ (see Remark 5.2.13 and Example 5.2.6 of \cite{FischerRuzhanskyBook2015}).  
\end{remark}
For an arbitrary graded Lie group, the H\"ormander classes $S^m_{\rho,\delta}(G\times \widehat{G}),$ $m\in\mathbb{R},$ provide a symbolic calculus closed under compositions, adjoints, and existence of parametrices. The following theorem summarises the composition and the adjoint rules for global operators  (see Theorem 5.5.3 and Theorem 5.5.12 of \cite{FischerRuzhanskyBook2015} for details about the existence of the corresponding  asymptotic expansions).
\begin{theorem}\label{calculus}
Let $A$ and $B$ be two pseudo-differential operators with symbols $\sigma\in S^{m}_{\rho,\delta}(G\times \widehat{G})$ and $\tau\in S^{r}_{\rho,\delta}(G\times \widehat{G}),$ where $m,r\in \mathbb{R},$ $0\leq \delta\leq \rho\leq1 ,$ and $\delta<1.$ Then the symbol of the composition operator $AB,$ belongs to  $S^{m+r}_{\rho,\delta}(G\times \widehat{G}). $ Moreover, the symbol of  the formal adjoint $A^*$ of $A,$ belongs to the class $ S^{m}_{\rho,\delta}(G\times \widehat{G}).$ 
\end{theorem}

In particular, the following theorem, which is an extension of the classical Calder\'on-Vaillancourt theorem, shows that operators with order zero and of $(\rho,\rho)$-type are bounded on $L^2(G).$ 
\begin{theorem}\label{CVT}
Let $G$ be a graded Lie group of homogeneous dimension $Q$ and let $0\leq   \rho < 1.$  Let $A:C^\infty(G)\rightarrow\mathscr{D}'(G)$ be a pseudo-differential operator with symbol $\sigma\in S^{0}_{\rho,\rho}(G\times \widehat{G} ).$ Then $A$ admits a bounded extension on $L^2(G),$  with the operator norm bounded, modulo by  a constant factor,  by a semi-norm $\Vert \sigma\Vert_{k,S^{0}_{\rho,\rho}}$ for some integer $k\in\mathbb{N}_0.$
\end{theorem}
 Theorem \ref{CVT} is indeed, Proposition 5.7.14 of \cite{FischerRuzhanskyBook2015}. A fundamental tool for its proof is Hulanicki Theorem (see \cite[page 251]{FischerRuzhanskyBook2015}) and the corresponding Calder\'on-Zygmund theory developed there. The integer $k$ in Theorem \ref{CVT} can be estimated by a proportional factor to $Q.$  Theorem
\ref{CVT} is sharp in the sense that for $\rho=1,$ and $G=\mathbb{R}^n,$ the class $S^0_{1,1}(\mathbb{R}^n\times \mathbb{R}^n)$ contains symbols whose pseudo-differential operators are unbounded on $L^p(\mathbb{R}^n)$ for all $1<p<\infty,$ (see Taylor \cite{Taylorbook1981}).
\begin{remark}\label{remarkoffinite}
If $\rho=0$ in Theorem \ref{CVT}, the operator norm of $A$ can be bounded by the supremum of all seminorms $p_{\alpha,\beta,0,0}(\sigma),$ where $ [\alpha] \leq rp_0,$ $ [\beta] \leq r\nu+[\frac{Q}{2}],$ $|\gamma|\leq r\nu,$ where $\nu$ is the degree of homogeneity of $\mathcal{R},$ $p_0/2$ is the smallest positive integer divisible by $\nu_1,\nu_2,\cdots ,\nu_n,$ and $r\in\mathbb{N}_0$ is the smallest integer such that $rp_0>Q+1,$ (see Proposition 5.7.7 of \cite{FischerRuzhanskyBook2015}).
\end{remark}
Theorem 5.7.1 of \cite{FischerRuzhanskyBook2015} consists of the non-commutative Calder\'on-Vaillancourt Theorem (Theorem \ref{CVT}) together with  the following end-point  case $\rho=1.$ 
\begin{theorem}\label{1delta}
Let $G$ be a graded Lie group of homogeneous dimension $Q$ and let  $0\leq \delta< 1.$  Let $A:C^\infty(G)\rightarrow\mathscr{D}'(G)$ be a pseudo-differential operator with symbol $\sigma\in S^{0}_{1,\delta}(G\times \widehat{G} ).$ Then $A$ admits a bounded extension on $L^2(G),$  with the operator norm bounded, modulo by  a constant factor,  by a semi-norm $\Vert \sigma\Vert_{k,S^{0}_{1,\delta}}$ for some integer $k\in\mathbb{N}_0.$
\end{theorem}

\subsection{$L^\infty$-$BMO$ and $H^1$-$L^1$ boundedness of pseudo-differential operators}

We will use the following variable change theorem on $\widehat{G},$ (see \cite{FischerRuzhanskyBook2015} or Fischer and Fermanian-Kammerer \cite[page 9]{FischerFermanian-Kammerer2017}).
\begin{equation}\label{elcambio}
\int_{\widehat{G}}{F}(r\cdot\pi)d\pi=r^{-Q}\int_{\widehat{G}}{F}(\pi)d\pi,
\end{equation}and in particular the following  useful consequence of \eqref{elcambio},
\begin{equation}\label{IdentityDilaG}
   \Vert F(r\cdot \pi)\Vert_{L^2(\widehat{G})}= r^{-\frac{Q}{2}}\Vert F( \pi)\Vert_{L^2(\widehat{G})},
\end{equation} where we have denoted by $L^2(\widehat{G}),$ the Hilbert-Schmidt-$L^2(\widehat{G})$ space defined as the completion of $\mathscr{S}(\widehat{G}),$ by the norm
\begin{equation}
    \Vert F(\pi) \Vert_{L^2(\widehat{G})}:=\left(\int\limits_{\widehat{G}}\Vert F(\pi)\Vert^2_{\textnormal{HS}}d\pi\right)^{\frac{1}{2}},\,\,\,F\in \mathscr{S}(\widehat{G}):=\mathscr{F}_G(\mathscr{S}(G)). 
\end{equation}
Let us recall that, for every $r>0,$ and $\pi\in\widehat{G},$ the representation $r\cdot\pi\in \widehat{G}$ is defined by the action $r\cdot\pi(x)=\pi(D_r(x)).$
  For our further analysis we  need the following technical lemma.
\begin{lemma}\label{compare}
Let $G$ be a graded Lie group and let $\mathcal{R}$ be a positive Rockland operator on $G.$ If $a,b\geq 0,$ then we have
\begin{equation}\label{ineqoper}
    \int\limits_{\widehat{G}}\Vert \pi(a+\mathcal{R})^{-b}\tau(\pi)\Vert_{\textnormal{HS}}^2d\pi\lesssim \int\limits_{\widehat{G}}\Vert \pi(\mathcal{R})^{-b}\tau(\pi)\Vert_{\textnormal{HS}}^2d\pi 
\end{equation} in the sense that if the integral in  right hand-side is finite, then the inequality \eqref{ineqoper} holds true.
\end{lemma}
\begin{proof}
It is sufficient to prove that for a.e. $\pi\in \widehat{G},$ 
\begin{equation}
     \Vert \pi(a+\mathcal{R})^{-b}\tau(\pi)\Vert_{\textnormal{HS}}^2\leq\Vert \pi(\mathcal{R})^{-b}\tau(\pi)\Vert_{\textnormal{HS}}^2.
\end{equation} For every $\pi\in \widehat{G},$ let us denote by $\{dE_{\pi(\mathcal{R})}(\lambda)\}_{\lambda>0}$ the spectral measure  associated to the operator $\pi(\mathcal{R}).$ If $B_{\pi}=\{e_{\pi,k}\}_{k=1}^\infty$ is a basis of the representation space $H_\pi,$ then,
\begin{align*}
    \Vert \pi(a+\mathcal{R})^{-b}\tau(\pi)\Vert_{\textnormal{HS}}^2&=\sum_{k=1}^\infty\left\Vert  \pi(a+\mathcal{R})^{-b} \tau(\pi)e_{\pi,k} \right\Vert_{H_\pi}^2\\
   &= \sum_{k=1}^\infty\left\Vert \int_{0}^\infty(a+\lambda)^{-b}dE_{\pi(\mathcal{R})}(\lambda)\tau(\pi)e_{\pi,k} \right\Vert_{H_\pi}^2\\
    &=\sum_{k=1}^\infty \int_{0}^\infty(a+\lambda)^{-2b}d\Vert E_{\pi(\mathcal{R})}(\lambda)\tau(\pi)e_{\pi,k} \Vert_{H_\pi}^2\\
    &\leq\sum_{k=1}^\infty \int_{0}^\infty\lambda^{-2b}d\Vert E_{\pi(\mathcal{R})}(\lambda)\tau(\pi)e_{\pi,k} \Vert_{H_\pi}^2\\
    &=\sum_{k=1}^\infty\left\Vert \int_{0}^\infty\lambda^{-b}dE_{\pi(\mathcal{R})}(\lambda)\tau(\pi)e_{\pi,k} \right\Vert_{H_\pi}^2\\
    &=\sum_{k=1}^\infty\left\Vert  \pi(\mathcal{R})^{-b} \tau(\pi)e_{\pi,k} \right\Vert_{H_\pi}^2\\
    &=\Vert \pi(\mathcal{R})^{-b}\tau(\pi)\Vert_{\textnormal{HS}}^2.
\end{align*}
Thus, the proof is complete.
\end{proof}
Now, we will  study the $L^\infty(G)$-boundedness for pseudo-differential operators with symbols absorbing projections compactly supported in the spectrum of $\mathcal{R}$ in the sense of the equality  \eqref{spectralabsorbsion}. This is a way of saying that the symbol of some pseudo-differential operator is compactly supported in the spectrum of $\mathcal{R},$ which in the Euclidean case agrees with the notion of symbols compactly supported in the frequency variables $\pi\in\widehat{G}$.

\begin{lemma}\label{FundamentallemmaI}
Let $G$ be a graded Lie group of homogeneous dimension $Q.$  Let $0\leq \delta,\rho\leq 1.$ Let $\sigma\in S^{-\frac{Q(1-\rho)}{2}}_{\rho,\delta}(G\times \widehat{G})$ be a symbol satisfying
\begin{equation}\label{spectralabsorbsion}
  \sigma(x,\pi)\psi_{j}(\pi(\mathcal{R}))=\sigma(x,\pi),
\end{equation} where $\psi_{j}(\lambda)=\psi_0(2^{-j}\lambda),$ for some test function $\psi_{0}\in C^\infty_0(0,\infty),$ and some fixed integer $j\in \mathbb{N}_0.$ Then $A=\textnormal{Op}(\sigma)$ extends to a bounded operator from $L^\infty(G)$ to $L^\infty(G),$ and for $\ell:=2\nu_0,$ where $\nu_0$ be the least common multiple of the weights  $\nu_1,\cdots,\nu_n,$ we have
\begin{equation*}
    \Vert  A\Vert_{\mathscr{B}(L^\infty(G))}\leq C \left(\sup_{\pi\in\widehat{G}, [\alpha_1] \leq \ell} \Vert [\Delta^{\alpha_1} \sigma(x,\pi)]\pi((1+\mathcal{R})^{\frac{ \frac{Q(1-\rho)}{2}+\rho [\alpha_1]  }{\nu}})  \Vert_{\textnormal{op}}     \right),
\end{equation*}
with the positive constant $C$  independent of $j$ and $\sigma.$ 
\end{lemma}
\begin{proof}
Let us fix $j\in \mathbb{N}_0.$  Let $b=R^{-\rho},$ $R=2^j.$ In view of Remark \ref{Linftyremark}, we only need to prove that  \begin{equation*}
    \sup_{x\in G}\Vert k_x \Vert_{L^1(G)}\leq C,
\end{equation*} where $C$ is a positive constant independent of $R,$ and $k_x,$ as usually, is the right-convolution kernel of $A.$ First, let us split the $L^1(G)$-norm of $k_x$ as,
\begin{equation*}
    \int\limits_{G}|k_x(z)|dz= \int\limits_{|z|\leq b}|k_x(z)|dz+ \int\limits_{|z|> b}|k_x(z)|dz.
\end{equation*} By using the H\"older inequality we estimate the first integral as follows,

\begin{align*}
    \int\limits_{|z|\leq b}|k_x(z)|dz&\leq \left(\int\limits_{|z|\leq b}dz\right)^{\frac{1}{2}}\left(\int\limits_{|z|\leq b}|k_x(z)|^2dz\right)^{\frac{1}{2}}\\
    &\asymp R^{\frac{Q(-\rho)}{2}}\Vert k_{x}\Vert_{L^2(G)}
    \\
    &= 2^{-\frac{j\rho Q}{2}}\Vert k_{x}\Vert_{L^2(G)}.
\end{align*}The Plancherel theorem   and the definition of the right-convolution kernel: $k_{x}=\mathscr{F}_{G}^{-1}(\sigma(x,\cdot)),$ for every $x\in G,$ imply

\begin{align*}
   \Vert k_{x}\Vert_{L^2(G)} &=\left( \int\limits_{\widehat{G}}\Vert \sigma(x,\pi) \Vert_{\textnormal{HS}}^2d\pi \right)^\frac{1}{2}\\&=\left( \int\limits_{\widehat{G}}\Vert \sigma(x,\pi)\pi(1+\mathcal{R})^{\frac{Q(1-\rho)}{2\nu}})\pi(1+\mathcal{R})^{-\frac{Q(1-\rho)}{2\nu}} \Vert_{\textnormal{HS}}^2d\pi \right)^\frac{1}{2}\\
    &=\left( \int\limits_{\widehat{G}}\Vert \sigma(x,\pi)\psi_{j}(\pi(\mathcal{R}))\pi(1+\mathcal{R})^{\frac{Q(1-\rho)}{2\nu}})\pi(1+\mathcal{R})^{-\frac{Q(1-\rho)}{2\nu}} \Vert_{\textnormal{HS}}^2d\pi \right)^\frac{1}{2}\\
     &= \left( \int\limits_{\widehat{G}}\Vert \sigma(x,\pi)\pi(1+\mathcal{R})^{\frac{Q(1-\rho)}{2\nu}}\psi_{j}(\pi(\mathcal{R}))\pi(1+\mathcal{R})^{-\frac{Q(1-\rho)}{2\nu}} \Vert_{\textnormal{HS}}^2d\pi \right)^\frac{1}{2}\\
     &\leq \left(\sup_{\pi\in\widehat{G}}\Vert \sigma(x,\pi)\pi(1+\mathcal{R})^{\frac{Q(1-\rho)}{2\nu}} \Vert_{\textnormal{op}}\right) \left( \int\limits_{\widehat{G}}\Vert\psi_{j}(\pi(\mathcal{R}))\pi(1+\mathcal{R})^{-\frac{Q(1-\rho)}{2\nu}} \Vert_{\textnormal{HS}}^2d\pi \right)^\frac{1}{2}.
\end{align*}
Let us denote, for every $j\geq 1,$
$$
  I_j:=  \left( \int\limits_{\widehat{G}}\Vert\psi_{j}(\pi(\mathcal{R}))\pi(1+\mathcal{R})^{-\frac{Q(1-\rho)}{2\nu}} \Vert_{\textnormal{HS}}^2d\pi \right)^\frac{1}{2}.
$$
By using the identity \eqref{IdentityDilaG} for $p=2,$ and the functional calculus, we obtain

\begin{align*}
 I_j^2
= \int\limits_{\widehat{G}} &\Vert \psi_{j}(\pi(\mathcal{R}))\pi((1+\mathcal{R})^{-\frac{Q(1-\rho)}{2\nu}}) \Vert_{\textnormal{HS}}^2d\pi \\
= \int\limits_{\widehat{G}}&\Vert\psi_0(2^{-j}\pi(\mathcal{R}))(1+\pi(\mathcal{R}))^{-\frac{Q(1-\rho)}{2\nu}} \Vert_{\textnormal{HS}}^2d\pi \\
  = \int\limits_{\widehat{G}}&\Vert\psi_0((2^{-\frac{j}{\nu}}\cdot\pi)(\mathcal{R}))(1+(2^{\frac{j}{\nu}}\cdot(2^{-\frac{j}{\nu}}\cdot\pi))(\mathcal{R}))^{-\frac{Q(1-\rho)}{2\nu}} \Vert_{\textnormal{HS}}^2d\pi.
  \end{align*} Now, from \eqref{IdentityDilaG} we deduce,
  \begin{align*}
 I_j^2 =2^{\frac{jQ}{\nu}} \int\limits_{\widehat{G}}&\Vert\psi_0(\pi(\mathcal{R}))(1+(2^{\frac{j}{\nu}}\cdot\pi)(\mathcal{R}))^{-\frac{Q(1-\rho)}{2\nu}} \Vert_{\textnormal{HS}}^2d\pi 
  \\=2^{\frac{jQ}{\nu}} \int\limits_{\widehat{G}}&\Vert\psi_0(\pi(\mathcal{R}))(1+2^{j}\pi(\mathcal{R}))^{-\frac{Q(1-\rho)}{2\nu}} \Vert_{\textnormal{HS}}^2d\pi \\
   =2^{\frac{jQ}{\nu}} \int\limits_{\widehat{G}}&\left\Vert\int\limits_{0}^\infty\psi_0(\lambda)(1+2^{j}\lambda)^{-\frac{Q(1-\rho)}{2\nu}} dE_{\pi(\mathcal{R})}(\lambda)\right\Vert_{\textnormal{HS}}^2d\pi.\end{align*}
   Consequently,
   \begin{align*}
  I_j^2 =2^{\frac{jQ}{\nu}-\frac{jQ(1-\rho)}{\nu}} \int\limits_{\widehat{G}} &\left\Vert\int\limits_{0}^\infty\psi_0(\lambda)(2^{-j}+\lambda)^{-\frac{Q(1-\rho)}{2\nu}} dE_{\pi(\mathcal{R})}(\lambda)\right\Vert_{\textnormal{HS}}^2d\pi \\
  \lesssim2^{\frac{jQ}{\nu}-\frac{jQ(1-\rho)}{\nu}} \int\limits_{\widehat{G}} &\left\Vert\int\limits_{0}^\infty\psi_0(\lambda)\lambda^{-\frac{Q(1-\rho)}{2\nu}} dE_{\pi(\mathcal{R})}(\lambda)\right\Vert_{\textnormal{HS}}^2d\pi \\
  = 2^{\frac{jQ\rho}{\nu}} \int\limits_{\widehat{G}}&\Vert\psi_0(\pi(\mathcal{R}))\pi(\mathcal{R})^{-\frac{Q(1-\rho)}{2\nu}} \Vert_{\textnormal{HS}}^2d\pi .
\end{align*}  
Note that the $L^2(\widehat{G})$-norm of  $\pi\mapsto H(\pi):=[\psi_{0}(\pi (\mathcal{R})) \pi(\mathcal{R})^{-\frac{Q(1-\rho)}{2\nu }}],$ is finite. Indeed, the smooth function  $\psi_{0}$ has compact support in $(0,\infty),$ and $\lambda=0$ is an isolated point for the spectrum of $H(\pi)$  (see Geller \cite{Geller1983} or \cite[Section 3.2.8]{FischerRuzhanskyBook2015}). So we  conclude that
\begin{equation*}
    I_j\lesssim 2^{\frac{jQ\rho}{2\nu}}=R^{\frac{ Q\rho}{2\nu}}.
\end{equation*}
This analysis  allows us to deduce that
\begin{align*}
     \left(\int\limits_{|z|\leq b}dz\right)^{\frac{1}{2}}\left(\int\limits_{|z|\leq b}|k_x(z)|^2dz\right)^{\frac{1}{2}}\asymp 2^{\frac{jQ(-\rho)}{2}}2^{\frac{ jQ\rho}{2\nu}}\leq2^{\frac{jQ(-\rho)}{2\nu}}2^{\frac{ jQ\rho}{2\nu}} =1,
\end{align*} and consequently we estimate
\begin{equation*}
    \int\limits_{|z|\leq b}|k_x(z)|dz=O(1).
\end{equation*}
To estimate the integral $\int\limits_{|z|> b}|k_x(z)|dz,$ we will use a suitable difference operator. Let $\nu_0$ be the least common multiple of $\nu_1,\cdots,\nu_n.$ Let $\Delta_q$ be  the difference operator associated to  $ q(x)=\sum_{j=1}^{n}x_{j}^\frac{2\nu_0}{\nu_j}.$ Then,
\begin{equation}
    \Delta_q:=\sum_{j=1}^{n}\Delta^{\alpha(j)},
\end{equation} where $\alpha(j)\in \mathbb{N}_0^n,$ is defined by $\alpha(j):=\frac{2\nu_0}{\nu_j}e_{j},$ with $\{e_{j}\}_{j=1}^n,$ being the canonical basis of $\mathbb{R}^n.$ Since, every difference operator $\Delta^{\alpha(j)},$ satisfies the Leibniz rule \eqref{differenceespcia;l}, the Leibniz rule for  $\Delta_q$ takes the form
\begin{equation}\label{differenceespcia;l22}
  \Delta_q(\tau_1\tau_2)=\sum_{j=1}^{n}\sum_{\alpha(j)_1+\alpha(j)_2=\alpha(j)}c_{\alpha(j)_1,\alpha(j)_2}\Delta^{\alpha(j)_1}(\tau_1)\Delta^{\alpha(j)_2}(\tau_2),\,\,\,\,\,\tau_1,\tau_2\in L^\infty_{a,b}(\widehat{G}).
\end{equation}

The function  $q$ is homogeneous of order $2\nu_0.$  So, for $\ell=2\nu_0$  we  observe that, 
\begin{align*}
    \int\limits_{|z|> b}|k_x(z)|dz &\leq \left(\int\limits_{|z|> b}q(z)^{-2}dz\right)^{\frac{1}{2}}\left(\int\limits_{|z|> b}|q(z)k_x(z)|^2dz\right)^{\frac{1}{2}}\\
    &\lesssim \left(\int\limits_{|z|> b}|z|^{-2\ell}dz\right)^{\frac{1}{2}}\left(\int\limits_{G}|q(z)k_x(z)|^2dz\right)^{\frac{1}{2}}\\
    &= b^{\frac{Q-2\ell}{2}}\left(\int\limits_{\widehat{G}}\Vert \Delta_q\sigma(x,\pi)\Vert^2_{\textnormal{HS}}d\pi\right)^{\frac{1}{2}}=b^{\frac{Q}{2}-\ell}\Vert \Delta_q \sigma(x,\pi)\Vert_{L^2(\widehat{G})}\\
    &=2^{j\rho(\frac{Q}{2}-\ell)}\Vert \Delta_q \sigma(x,\pi)\Vert_{L^2(\widehat{G})}.  
\end{align*}
Denoting $\psi_j(\pi)=\psi_j(\pi(\mathcal{R})),$ and using  the Leibniz rule \eqref{differenceespcia;l22},  we can find a finite family of difference operators $\Delta^{\alpha(j)_1}$ and $\Delta^{\alpha(j)_2}$, associated to the functions $q_{\alpha(j)_i}=x^{\alpha_i},$ for every pair $(\alpha(j)_1,\alpha(j)_2)$ such that $\alpha(j)_1+\alpha(j)_2=\alpha(j),$ and  we can estimate 
\begin{align*}
   & \Vert \Delta_q \sigma(x,\pi)\Vert_{L^2(\widehat{G})}  =\Vert \Delta_q [\sigma(x,\pi)\psi_{j}(\pi)]\Vert_{L^2(\widehat{G})}\\
    &\leq\sum_{k=1}^n \sum_{\alpha(k)_1+\alpha(k)_2=\alpha(k)}C_{\alpha(k)_1,\alpha(k)_2}\Vert [\Delta^{\alpha(k)_1} \sigma(x,\pi)][[\Delta^{\alpha(k)_2}\psi_{j}](\pi)]\Vert_{L^2(\widehat{G})}\\
     &\leq\sum_{k=1}^n \sum_{\alpha(k)_1+\alpha(k)_2=\alpha(k)}C_{\alpha(k)_1,\alpha(k)_2}\Vert [\Delta^{\alpha(k)_1} \sigma(x,\pi)]\pi((1+\mathcal{R})^{\frac{ \frac{Q(1-\rho)}{2}+\rho [\alpha(k)_1]  }{\nu}}) \\
    & \hspace{5cm}\pi((1+\mathcal{R})^{-\frac{ \frac{Q(1-\rho)}{2}+\rho [\alpha(k)_1]   }{\nu}}) [\Delta^{\alpha(k)_2}\psi_{j}](\pi))\Vert_{L^2(\widehat{G})}\\
    &\lesssim\left(\sup_{\pi\in\widehat{G}, [\alpha_1] \leq \ell} \Vert [\Delta^{\alpha_1} \sigma(x,\pi)]\pi((1+\mathcal{R})^{\frac{ \frac{Q(1-\rho)}{2}+\rho [\alpha_1]  }{\nu}})  \Vert_{\textnormal{op}}     \right)  \\
    &\hspace{2cm}\times\sum_{k=1}^n\sum_{\alpha(k)_1+\alpha(k)_2=\alpha(k)}\Vert\pi((1+\mathcal{R})^{-\frac{ \frac{Q(1-\rho)}{2}+\rho [\alpha(k)_1]  }{\nu}}) [\Delta^{\alpha(k)_2}\psi_{j}](\pi)\Vert_{L^2(\widehat{G})}.
\end{align*}For simplicity, let us denote
\begin{equation*}
    f_0:=\mathscr{F}_{G}^{-1}(\psi_0(\pi(\mathcal{R})),\,\,\,\sigma_{0}=\widehat{f}_0.
\end{equation*} Then,
\begin{equation*}
  \textnormal{  for  }  \sigma_{0_{(r)}}:=\{\sigma_{0}(r\cdot \pi):\pi\in\widehat{G}\}  ,\,\,\,f_{0_{(r)}}:=r^{-Q}f_0\circ D_{r}\textnormal{  we have  }\widehat{f}_{0_{(r)}}=\sigma_{0_{(r)}},
\end{equation*}for every $r>0.$
In particular, for $r=2^{-\frac{j}{\nu}},$
\begin{align*}
   \psi_j(\pi)\equiv\psi_{j}(\pi(\mathcal{R}))&=\psi_0(2^{-j}\pi(\mathcal{R}))= \psi_0((2^{-\frac{j}{\nu}}\cdot \pi)(\mathcal{R}))=\sigma_0(r\cdot \pi)\\
   &\equiv \sigma_{0_{(r)}}(\pi).
\end{align*}
By using the action of difference operators on the dilations of representations in the unitary dual (see \eqref{FischerFermanian-Kammerer2017}), we have
\begin{equation*}
    [\Delta^{\alpha(k)_2}\psi_{j}](\pi)=[\Delta^{\alpha(k)_2} \sigma_{0_{(r)}}](\pi)=r^{  [\alpha(k)_2]  }[\Delta^{\alpha(k)_2} \sigma_{0}](r\cdot \pi).
\end{equation*}
By keeping the notation $r=2^{-\frac{j}{\nu}},$ we have
\begin{align*}
     &\Vert\pi((1+\mathcal{R})^{-\frac{ \frac{Q(1-\rho)}{2}+\rho [\alpha(k)_1]   }{\nu}}) [\Delta^{\alpha(k)_2}\psi_{j}](\pi)\Vert_{L^2(\widehat{G})}\\&=\Vert(1+\pi(\mathcal{R}))^{-\frac{ \frac{Q(1-\rho)}{2}+\rho [\alpha(k)_1]   }{\nu}})     r^{  [\alpha(k)_2]  }[\Delta^{\alpha(k)_2} \sigma_{0}](r\cdot \pi)  \Vert_{L^2(\widehat{G})}
    \\
    &=\Vert (1+(r^{-1}\cdot r\cdot \pi)(\mathcal{R}))^{-\frac{ \frac{Q(1-\rho)}{2}+\rho [\alpha(k)_1]   }{\nu}}) r^{  [\alpha(k)_2]  }[\Delta^{\alpha(k)_2}\sigma_{0}](r\cdot \pi)\Vert_{L^2(\widehat{G})}
     \\&=r^{-\frac{Q}{2}+  [\alpha(k)_2]  }\Vert (1+(r^{-1}\cdot  \pi)(\mathcal{R}))^{-\frac{ \frac{Q(1-\rho)}{2}+\rho [\alpha(k)_1]   }{\nu}}) [\Delta^{\alpha(k)_2}\sigma_{0}]( \pi) \Vert_{L^2(\widehat{G})},
\end{align*}
where in the last line we have used the identity 
\eqref{IdentityDilaG} for $p=2.$ Again, by using the functional calculus we have
\begin{align*}
 &  \Vert (1+(r^{-1}\cdot   \pi)(\mathcal{R}))^{-\frac{ \frac{Q(1-\rho)}{2}+\rho [\alpha(k)_1]   }{\nu}}) [\Delta^{\alpha(k)_2}\sigma_{0}]( \pi) \Vert_{L^2(\widehat{G})}\\
   &=\Vert (1+  \pi(r^{-\nu}\mathcal{R}))^{-\frac{ \frac{Q(1-\rho)}{2}+\rho [\alpha(k)_1]   }{\nu}}) [\Delta^{\alpha(k)_2}\sigma_{0}]( \pi) \Vert_{L^2(\widehat{G})}\\
    &=\left\Vert \int\limits_{0}^\infty(1+(r^{-\nu}\lambda))^{-\frac{ \frac{Q(1-\rho)}{2}+\rho [\alpha(k)_1]   }{\nu}})dE_{\pi(\mathcal{R})}(\lambda) [\Delta^{\alpha(k)_2}\sigma_{0}]( \pi) \right\Vert_{L^2(\widehat{G})}\\
    &= r^{  \frac{Q(1-\rho)}{2}+\rho [\alpha(k)_1]    }   \left\Vert \int\limits_{0}^\infty(r^{\nu}+\lambda)^{-\frac{ \frac{Q(1-\rho)}{2}+\rho [\alpha(k)_1]   }{\nu}}dE_{\pi(\mathcal{R})}(\lambda) [\Delta^{\alpha(k)_2}\sigma_{0}]( \pi) \right\Vert_{L^2(\widehat{G})}.
\end{align*}
Consequently, we obtain
\begin{align*}
 &\Vert\pi((1+\mathcal{R})^{-\frac{ \frac{Q(1-\rho)}{2}+\rho [\alpha(k)_1]   }{\nu}}) [\Delta^{\alpha(k)_2}\psi_{j}(\pi(1+\mathcal{R}))^{\frac{1}{\nu}}]\Vert_{L^2(\widehat{G})}\\
  &=r^{ -\frac{Q}{2}+  [\alpha(k)_2]  + \frac{Q(1-\rho)}{2}+\rho [\alpha(k)_1]    }  \Vert \pi(r^{\nu}+\mathcal{R})^{-\frac{ \frac{Q(1-\rho)}{2}+\rho [\alpha(k)_1]   }{\nu}} [\Delta^{\alpha(k)_2}\sigma_{0}]( \pi)]]\Vert_{L^2(\widehat{G})} 
  \\
  &=r^{ -\frac{Q}{2}+  [\alpha(k)_2]  + \frac{Q(1-\rho)}{2}+\rho [\alpha(k)_1]    }  N_{\alpha(k)_1,\alpha(k)_2},
\end{align*}
   where $N_{\alpha(k)_1,\alpha(k)_2}$ satisfies
   \begin{align*}
     & \Vert \pi(r^{\nu}+\mathcal{R})^{-\frac{ \frac{Q(1-\rho)}{2}+\rho [\alpha(k)_1]   }{\nu}} [\Delta^{\alpha(k)_2}\sigma_{0}]( \pi)\Vert_{L^2(\widehat{G})} \\
     & \Vert \pi(2^{-j\nu}+\mathcal{R})^{-\frac{ \frac{Q(1-\rho)}{2}+\rho [\alpha(k)_1]   }{\nu}} [\Delta^{\alpha(k)_2}\sigma_{0}]( \pi)\Vert_{L^2(\widehat{G})} \\
    &\lesssim  \Vert \pi(\mathcal{R})^{-\frac{ \frac{Q(1-\rho)}{2}+\rho [\alpha(k)_1]   }{\nu}} [\Delta^{\alpha(k)_2}\sigma_{0}]( \pi)\Vert_{L^2(\widehat{G})}
 =:  N_{\alpha(k)_1,\alpha(k)_2} <\infty,
   \end{align*} where in the last line, we have used  \eqref{ineqoper} with $a=2^{-j\nu}$ and $b=\frac{ \frac{Q(1-\rho)}{2}+\rho [\alpha(k)_1]   }{\nu}$.
Observe that  $N_{\alpha(k)_1,\alpha(k)_2}$ is a finite number because it is the $L^2(\widehat{G})$ norm of the  function $$G_{\alpha(k)_1,\alpha(k)_2}(\pi):= \pi(\mathcal{R})^{-\frac{ \frac{Q(1-\rho)}{2}+\rho [\alpha(k)_1]   }{\nu}} [\Delta^{\alpha(k)_2}\sigma_{0}]( \pi),$$ defined on the unitary dual $\widehat{G}.$ That the $L^2(\widehat{G})$-norm of $[\Delta^{\alpha(k)_2}\sigma_{0}]$ is finite can be justified because $[\Delta^{\alpha(k)_2}\sigma_{0}]=\mathscr{F}^{-1}(x^{\alpha(k)_2}f_0)$ and $f_0=\mathscr{F}_{G}^{-1}(\psi_0(\pi(\mathcal{R})),$ with $\phi_0\in C^\infty_0(\mathbb{R}_0^+).$  So, we can estimate
\begin{align*}
    \Vert \Delta_q \sigma(x,\pi)\Vert_{L^2(\widehat{G})} \leq C  \sum_{k=1}^n\sum_{\alpha(k)_1+\alpha(k)_2=\alpha(k)}2^{-\frac{j}{\nu}( -\frac{Q}{2}+  [\alpha(k)_2]  + \frac{Q(1-\rho)}{2}+\rho [\alpha(k)_1]    )},
\end{align*}
with
\begin{equation*}
    C=\left(\sup_{\pi\in\widehat{G}, [\alpha(k)_1] \leq \ell} \Vert [\Delta^{\alpha(k)_1} \sigma(x,\pi)]\pi((1+\mathcal{R})^{\frac{ \frac{Q(1-\rho)}{2}+\rho [\alpha(k)_1]  }{\nu}})  \Vert_{\textnormal{op}}     \right).
\end{equation*}
Taking into account that
\begin{align*}
    2^{-\frac{j}{\nu}( -\frac{Q}{2}+  [\alpha(k)_2]  + \frac{Q(1-\rho)}{2}+\rho [\alpha(k)_1]    )}&=2^{-\frac{j}{\nu}( -\frac{Q}{2}\rho+\ell - [\alpha(k)_1]  (1-\rho)    )}\leq 2^{-\frac{j}{\nu}( -\frac{Q}{2}\rho+\ell -\ell (1-\rho)   )}\\&= 2^{-\frac{j}{\nu}( -\frac{Q}{2}\rho+\ell\rho    )}=2^{-\frac{j}{\nu}\rho( \ell -\frac{Q}{2}    )}\\
    &=2^{\frac{j}{\nu}\rho( \frac{Q}{2}-\ell    )},
\end{align*}
the preceding  analysis allows us to conclude that
\begin{align*}
     \int\limits_{|z|> b}|k_x(z)|dz &\lesssim  2^{j(-\rho)(\frac{Q}{2}-\ell)}\Vert \Delta_q \sigma(x,\pi)\Vert_{L^2(\widehat{G})}\lesssim  2^{\frac{j}{\nu}(-\rho)(\frac{Q}{2}-\ell)}\Vert \Delta_q \sigma(x,\pi)\Vert_{L^2(\widehat{G})},\\
    &\lesssim 2^{\frac{j}{\nu}(-\rho)(\frac{Q}{2}-\ell)}\times  2^{\frac{j}{\nu}\rho( \frac{Q}{2}-\ell    )}=1.
\end{align*}
Thus, the proof is complete.
\end{proof}
The following Lemma \ref{FundamentallemmaII}, will be useful in order to control the seminorms of the Littlewood-Paley decomposition applied to the symbol $\sigma,$ and in its proof we will use the notation $\pi':=s^{\frac{1}{\nu}}\cdot\pi\in \widehat{G},$ $s>0,$ for the respective change of variables on the unitary dual.

\begin{lemma}\label{FundamentallemmaII}
Let $G$ be a graded Lie group of homogeneous dimension $Q,$ and let $\varepsilon>0.$ Let $0\leq \delta,\rho\leq 1.$ Let $\sigma\in S^{-\varepsilon}_{\rho,\delta}(G\times \widehat{G}).$ Let $\eta$ be a smooth function supported in $\{\lambda:R\leq\lambda\leq 3R\},$ for some $R>1.$ Then for all $\alpha\in\mathbb{N}_0^n$ with $ [\alpha] \leq \ell,$ there exists $C>0,$ such that for every $s>0,$ we have 
\begin{equation*}
   \sup_{(x,\pi)\in G\times \widehat{G}}  \Vert  \pi((1+\mathcal{R})^{\frac{ \varepsilon+\rho [\alpha]  }{\nu}})\Delta^{\alpha} [\sigma(x,\pi)\eta(s\pi(\mathcal{R}))]     \Vert_{\textnormal{op}}\leq C     \Vert \sigma\Vert_{\ell,S^{-\varepsilon}_{\rho,\delta}}s^{\frac{\ell}{\nu}},
\end{equation*}
with the positive constant $C$  independent of $s,$ $R$ and $\sigma.$ 
\end{lemma}
\begin{proof}
For the proof, we need to check that
\begin{equation}\label{eq:ref:rev}
 \sup_{1>s>0}s^{-\frac{\ell}{\nu}}  \sup_{(x,\pi')\in G\times \widehat{G}}  \Vert  \pi'((1+\mathcal{R})^{\frac{ \varepsilon+\rho [\alpha]  }{\nu}})\Delta^{\alpha} [\sigma(x,\pi')\eta(s\pi'(\mathcal{R}))]     \Vert_{\textnormal{op}}\lesssim     \Vert \sigma\Vert_{\ell,S^{-\varepsilon}_{\rho,\delta}}.
\end{equation} 
Indeed, in the case $s>1,$ one has 
\begin{align*}
& \sup_{s>1}s^{-\frac{\ell}{\nu}}  \sup_{(x,\pi')\in G\times \widehat{G}}  \Vert  \pi'((1+\mathcal{R})^{\frac{ \varepsilon+\rho [\alpha]  }{\nu}})\Delta^{\alpha} [\sigma(x,\pi')\eta(s\pi'(\mathcal{R}))]     \Vert_{\textnormal{op}}\\
&\lesssim 
\sup_{s>1}  \sup_{(x,\pi')\in G\times \widehat{G}}  \Vert  \pi'((1+\mathcal{R})^{\frac{ \varepsilon+\rho [\alpha]  }{\nu}})\Delta^{\alpha} [\sigma(x,\pi')\eta(s\pi'(\mathcal{R}))]     \Vert_{\textnormal{op}}\lesssim 
\Vert \sigma\Vert_{\ell,S^{-\varepsilon}_{\rho,\delta}},
\end{align*} because the symbol $\sigma(x,\pi')\eta(s\pi'(\mathcal{R}))$ belongs uniformly in $s>1,$ to the class $S^{-\varepsilon}_{\rho,\delta}(G\times \widehat{G}).$ 
Now we are going to prove \eqref{eq:ref:rev}. Using the change $\pi':=s^{\frac{1}{\nu}}\cdot \pi,$ we have
\begin{align*}
    &\sup_{1>s>0}s^{-\frac{\ell}{\nu}}  \sup_{(x,\pi')\in G\times \widehat{G}}  \Vert  \pi'((1+\mathcal{R})^{\frac{ \varepsilon+\rho [\alpha]  }{\nu}})\Delta^{\alpha} [\sigma(x,\pi')\eta(s\pi'(\mathcal{R}))]     \Vert_{\textnormal{op}}\\
    &=\sup_{1>s>0}s^{-\frac{\ell}{\nu}}  \sup_{(x,\pi)\in G\times \widehat{G}}  \Vert  (s^{\frac{1}{\nu}}\cdot\pi)((1+\mathcal{R})^{\frac{ \varepsilon+\rho [\alpha]  }{\nu}})\Delta^{\alpha} [\sigma(x,s^{\frac{1}{\nu}}\cdot\pi)\eta(s((s^{\frac{1}{\nu}}\cdot\pi(\mathcal{R})) )]     \Vert_{\textnormal{op}}\\
     &=\sup_{1>s>0}s^{-\frac{\ell}{\nu}}  \sup_{(x,\pi)\in G\times \widehat{G}}  \Vert  (s^{\frac{1}{\nu}}\cdot\pi)((1+\mathcal{R})^{\frac{ \varepsilon+\rho [\alpha]  }{\nu}})\Delta^{\alpha} [\sigma(x,s^{\frac{1}{\nu}}\cdot\pi)\eta(s^{\frac{2}{\nu}}\cdot\pi(\mathcal{R})) )]     \Vert_{\textnormal{op}}.
\end{align*}
The Leibniz rule allows us to write
\begin{align*}
  &\Delta^{\alpha} [\sigma(x,s^{\frac{1}{\nu}}\cdot\pi)\eta((s^{\frac{2}{\nu}}\cdot\pi)(\mathcal{R}))]\\
   &=\sum_{\alpha_1+\alpha_2=\alpha}C_{\alpha_1,\alpha_2}[ \Delta^{\alpha_1}\sigma(x,(s^{\frac{1}{\nu}}\cdot \pi))][\Delta^{\alpha_2}{\sigma_{\eta(\mathcal{R})}}_{(s^{\frac{2}{\nu}}\cdot)}](\pi),
\end{align*}  where ${\sigma_{\eta(\mathcal{R})}}_{s^{\frac{1}{\nu}}\cdot}=\{\sigma_{\eta(\mathcal{R})}(s^{\frac{1}{\nu}}\cdot\pi)\},$ is defined as in \eqref{eeeeeeeee}, in this case in terms of the symbol $\sigma_{\eta(\mathcal{R})}(\cdot).$   
In view of the action of difference operators on dilations of representations (see \eqref{FischerFermanian-Kammerer2017}) we have
\begin{align}
    [\Delta^{\alpha_2}{\sigma_{\eta(\mathcal{R})}}_{(s^{\frac{2}{\nu}}\cdot)}](\pi)=s^{\frac{2  [\alpha_2]  }{\nu}}[\Delta^{\alpha_2}\sigma_{\eta(\mathcal{R})}](s^{\frac{2}{\nu}}\cdot \pi),
\end{align}
and 
\begin{equation*}
  \Delta^{\alpha_1}[\sigma(x,(s^{\frac{1}{\nu}}\cdot \pi))]=s^{\frac{ [\alpha_1] }{\nu}}[\Delta^{\alpha_1}\sigma](x,(s^{\frac{1}{\nu}}\cdot \pi))], 
\end{equation*} 
and we deduce
\begin{align*}
     & \Vert (s^{\frac{1}{\nu}}\cdot \pi)((1+\mathcal{R})^{\frac{ \varepsilon+\rho [\alpha_1]  }{\nu}}) \Delta^{\alpha}[\sigma(x,(s^{\frac{1}{\nu}}\cdot \pi))\eta(s^{\frac{2}{\nu}}\cdot\pi(\mathcal{R}))]\Vert_{\textnormal{op}} \\
     &\leq 
       \sum_{\alpha_1+\alpha_2=\alpha}C_{\alpha_1,\alpha_2} \Vert (s^{\frac{1}{\nu}}\cdot \pi)((1+\mathcal{R})^{\frac{ \varepsilon+\rho [\alpha_1]  }{\nu}})\\
       &\hspace{4cm}\circ s^{\frac{  [\alpha_1] }{\nu}   }[ \Delta^{\alpha_1}\sigma](x,(s^{\frac{1}{\nu}}\cdot \pi))\Vert_{\textnormal{op}}s^{\frac{2  [\alpha_2]  }{\nu}}\Vert \Delta^{\alpha_2}\sigma_{\eta(\mathcal{R})}(s^{\frac{2}{\nu}}\cdot \pi)\Vert_{\textnormal{op}}.
\end{align*}Observing that $\ell\leq  [\alpha_1] +2  [\alpha_2]  \leq 2\ell,$ and that we need to estimate the  operator norm for $0<s<1,$ we have that $s^{\frac{ [\alpha_1] +2  [\alpha_2]  }{\nu}}\leq s^{\frac{\ell}{\nu}},$ that togheter with the estimate,
\begin{equation*}
    \sup_{s'>0}\Vert \Delta^{\alpha_2}\sigma_{\eta(\mathcal{R})}({s'}^{\frac{2}{\nu}}\cdot \pi)\Vert_{\textnormal{op}}\leq \sup_{\pi\in \widehat{G}}\Vert \Delta^{\alpha_2}\sigma_{\eta(\mathcal{R})}( \pi)\Vert_{\textnormal{op}}=\sup_{\pi\in \widehat{G}}\Vert \Delta^{\alpha_2}{\eta(\pi(\mathcal{R}))}\Vert_{\textnormal{op}}<\infty,
\end{equation*}
allow us to obtain that
\begin{align*}
   &\Vert \pi((1+\mathcal{R})^{\frac{ \varepsilon+\rho [\alpha]  }{\nu}})\Delta^{\alpha} [\sigma(x,\pi)\eta(s\pi(\mathcal{R}))]      \Vert_{\textnormal{op}}\lesssim  \Vert \sigma\Vert_{\ell,S^{-\varepsilon}_{\rho,\delta}} s^{\frac{\ell}{\nu}},
\end{align*} proving  Lemma \ref{FundamentallemmaII}.
\end{proof}

Now, we proceed with the following local estimate for symbols in global H\"ormander classes.
\begin{lemma}\label{LemmaJulio} Let $r>0,$ and $0\leq \delta\leq \rho\leq  1,$ $\delta\neq 1.$
Let  $ \tau\in S^{-\varepsilon}_{\rho,\delta}(G\times\widehat{G}),$ where $\varepsilon\geq0,$ and let  $T=\textnormal{Op}(\tau)$. If $\phi$ is a smooth  compactly supported real-valued function in $B(x_0,2r)$ satisfying that  $
    \phi(x)=1, \textnormal{  for  }\,\,x\in B(x_0,r),\textnormal{ and }0\leq \phi\leq 10,
$  there exists a positive  constant $C>0,$ independent of $r>0,$ such that
\begin{equation}\label{LE}
   I:=  \frac{1}{|B(x_0,r)|}\int\limits_{B(x_0,r)}|T[\phi f](x)|dx\leq  {  C\Vert \sigma_{TL}\Vert_{k,S^{0}_{\rho,\delta}} }\Vert f\Vert_{L^\infty(G)},
\end{equation}for some $k\in \mathbb{N}_0,$ and where $L:=(1+\mathcal{R})^{\frac{\varepsilon}{\nu}}.$
\end{lemma}
\begin{proof}
From the properties of $\phi$ we have
\begin{equation*}
 \Small{   \int\limits_{B(x_0,r)}\phi(x)^2dx\leq  \int\limits_{B(x_0,2r)}\phi(x)^2dx =\Vert \phi\Vert_{L^2(G)}^2\leq 100  |B(x_0,2r)|.}
\end{equation*} Consequently, we deduce
\begin{equation}\label{35}
      10|B(x_0,2r)|^{\frac{1}{2}}\leq 10C|B(x_0,r)|^{\frac{1}{2}},
\end{equation}where in the last inequality we have used  that the measure on the group satisfies the doubling property. To estimate $I,$ observe that, in view of the Cauchy-Schwarz inequality, we have
\begin{equation*}
    \frac{1}{|B(x_0,r)|}\int\limits_{B(x_0,r)}|T[\phi f](x)|dx\leq \frac{1}{|B(x_0,r)|^{\frac{1}{2}}}\left(\int\limits_{B(x_0,r)} |T[\phi f](x)|^2dx   \right)^{\frac{1}{2}}.
\end{equation*}
Let $L:=(1+\mathcal{R})^{\frac{\varepsilon}{\nu}}\in S^{\varepsilon}_{1,0}(G\times\widehat{G})\subset S^{\varepsilon}_{\rho,0}(G\times\widehat{G})\subset S^{\varepsilon}_{\rho,\delta}(G\times\widehat{G}).$ Since $T\in\textnormal{Op}( S^{-\varepsilon}_{\rho,\delta}(G\times\widehat{G})), $ Theorem 5.2.22, part (ii) in \cite{FischerRuzhanskyBook2015} gives
\begin{equation*}
    TL=T(1+\mathcal{R})^{\frac{\varepsilon}{\nu}}\in S^{0}_{\rho,\delta}(G\times\widehat{G}).
\end{equation*}
In view of the condition $0\leq \delta\leq \rho\leq 1,$ $\delta\neq 1,$ the Calder\'on-Vaillancourt Theorem (see Proposition 5.7.14 of \cite{FischerRuzhanskyBook2015}) if $0\leq \delta\leq \rho<1,$ or Theorem \ref{1delta} if $0\leq \delta<\rho=1,$ implies that $ TL$ is bounded on $L^2(G),$  with the operator norm bounded, modulo   a constant factor,  by a semi-norm $\Vert \sigma_{TL}\Vert_{k,S^{0}_{\rho,\delta}}$  of the symbol of $TL,$ $\sigma_{TL}.$ Consequently,
\begin{align*}
 &  \frac{1}{|B(x_0,r)|^{\frac{1}{2}}}\left(\int\limits_{B(x_0,r)} |TL[L^{-1}(\phi f)](x)|^2dx   \right)^{\frac{1}{2}}\leq \frac{\Vert TL[L^{-1}(\phi f)] \Vert_{L^2(G)}}{|B(x_0,r)|^{\frac{1}{2}}}\\
   &\leq \frac{  C\Vert \sigma_{TL}\Vert_{k,S^{0}_{\rho,\delta}}    \Vert L^{-1}( \phi f )\Vert_{L^2(G)}}{|B(x_0,r)|^{\frac{1}{2}}}.
\end{align*}
By observing that 
\begin{equation*}
    \Vert L^{-1}(\phi f) \Vert_{L^2(G)}=\Vert \phi f\Vert_{H^{-\varepsilon,\mathcal{R}}(G)},
\end{equation*}
where $H^{-\varepsilon,\mathcal{R}}(G)$ is the Sobolev space of order $-\varepsilon,$ associated to $\mathcal{R},$ the embedding $L^2(G)\hookrightarrow H^{-\varepsilon,\mathcal{R}}(G), $ implies that
\begin{equation*}
 \Vert L^{-1}(\phi f) \Vert_{L^2(G)}=\Vert \phi f\Vert_{H^{-\varepsilon,\mathcal{R}}(G)}\lesssim \Vert \phi f \Vert_{L^2(G)} .    
\end{equation*}
Moreover, from \eqref{35}, we deduce the inequality
\begin{equation*}
     \Vert \phi f \Vert_{L^2(G)}\leq \Vert f\Vert_{L^\infty(G)}\Vert \phi\Vert_{L^2(G)}\lesssim   10\Vert f\Vert_{L^\infty(G)}|B(x_0,r)|^{\frac{1}{2}}.
\end{equation*}
So, we conclude
\begin{equation*}
    I:= \frac{1}{|B(x_0,r)|}\int\limits_{B(x_0,r)}|T[\phi f](x)|dx\leq {  C\Vert \sigma_{TL}\Vert_{k,S^{0}_{\rho,\delta}} }\Vert f\Vert_{L^\infty(G)},
\end{equation*}completing the proof.
\end{proof}

The following Theorem \ref{LinftyBMOCardonaDelgadoRuzhansky}, corresponds to Part (a) in our main Theorem \ref{MainTheorem}.

\begin{theorem}\label{LinftyBMOCardonaDelgadoRuzhansky}
Let $G$ be a graded Lie group of homogeneous dimension $Q.$ Let $A:C^\infty(G)\rightarrow\mathscr{D}'(G)$ be a pseudo-differential operator with symbol $\sigma\in S^{-m}_{\rho,\delta}(G\times \widehat{G} ),$ $0\leq \delta\leq \rho\leq 1,$ $\delta\neq 1.$ Then $A=\textnormal{Op}(\sigma)$ extends to a bounded operator from $L^\infty(G)$ to $BMO(G),$ and  we have
\begin{equation}\label{LBMO}
    \Vert  A\Vert_{\mathscr{B}(L^\infty(G),BMO(G))}\leq C\max\{ \Vert\sigma\Vert_{\ell,\,S^{-\frac{Q(1-\rho)}{2}}_{\rho,\delta}},    \Vert\sigma_{A(1+\mathcal{R})^{\frac{Q(1-\rho)}{2\nu}}}\Vert_{\ell,\,S^{0}_{\rho,\delta}}  \},
\end{equation}for $\ell\in\mathbb{N}$ large enough. Moreover, $A$ also extends to a bounded operator from the Hardy space $H^1(G)$ into $L^1(G)$ and 
\begin{equation}\label{H1L1}
    \Vert  A\Vert_{\mathscr{B}(H^1(G),L^1(G)))}\leq  C\max\{ \Vert\sigma^*\Vert_{\ell,\,S^{-\frac{Q(1-\rho)}{2}}_{\rho,\delta}},    \Vert\sigma_{A^*(1+\mathcal{R})^{\frac{Q(1-\rho)}{2\nu}}}\Vert_{\ell,\,S^{0}_{\rho,\delta}}  \},
\end{equation}where $\sigma^{*}\in S^{-\frac{Q(1-\rho)}{2}}_{\rho,\delta}(G\times \widehat{G})$ denotes the symbol of the formal adjoint $A^*.$
\end{theorem} 
\begin{proof}
 Let us fix $f\in L^\infty(G)$ and a ball $B(x_0,r)$ where $x_0\in G.$ We will prove that the estimate
\begin{equation*}
    \frac{1}{|B(x_0,r)|}\int\limits_{B(x_0,r)}|Af(x)-(Af)_{B(x_0,r)}|dx\leq C\Vert \sigma\Vert_{\ell,S^{-\frac{Q(1-\rho)}{2}}_{\rho,\delta}}\Vert f\Vert_{L^\infty(G)}
\end{equation*}
holds true with a positive constant $C>0$ independent of $f$ and $r,$ where 
\begin{equation*}
    (Af)_{B(x_0,r)}:=\textnormal{Average}(Af,B(x_0,r))=\frac{1}{|B(x_0,r)|}\int\limits_{B(x_0,r)}Af(x)dx
\end{equation*} and $\Vert \sigma\Vert_{\ell,S^{-\frac{Q(1-\rho)}{2}}_{\rho,\delta}}$ is the seminorm of $\sigma$
in the right hand side of \eqref{LBMO}. By using the spectral decomposition of $\mathcal{R},$ for a.e. $(x,\pi)\in  G\times\widehat{G},$ we will express  $\sigma(x,\pi)$ as the sum of two densely defined operators on $H_\pi,$
\begin{equation*}
    \sigma(x,\pi)=\sigma^0(x,\pi)+\sigma^1(x,\pi),\,\,\,\sigma^{j}(x,\pi):H_\pi^\infty\rightarrow H_\pi,\,j=0,1,
\end{equation*} in a such way that both, $\sigma^0(x,\pi)$  and  $\sigma^1(x,\pi),$ define two $\widehat{G}$-fields of operators satisfying
\begin{equation*}
    \Vert \sigma^{j}\Vert_{\ell,S^{-\frac{Q(1-\rho)}{2}}_{\rho,\delta}}\leq C_{j,\ell}\Vert \sigma\Vert_{\ell,S^{-\frac{Q(1-\rho)}{2}}_{\rho,\delta}},\,\,j=0,1,\,\ell\geq 1. 
\end{equation*}
To guarantee the existence of $\sigma^j,$ let us consider $\gamma\in C^\infty_0(\mathbb{R}, \mathbb{R}^+_0),$ satisfying the following requirements: for $|t|\leq \frac{1}{2},$ $\gamma(t)=1$ and $\gamma(t)=0$ for all $t$ with $|t|\geq 1.$
Let us define, for a.e. $\pi\in \widehat{G}$ the operator
\begin{equation*}
    \tilde{\gamma}(\pi):=\gamma(\frac{r^{-\nu}}{\lambda^\nu_{\mathcal{R}} }\pi(\mathcal{R})):H_\pi\rightarrow H_\pi,
\end{equation*}
where $\lambda_{\mathcal{R}}>0,$ is a positive real number which will be defined later. Let us define,
\begin{equation}\label{gamasupport}
    \sigma^0(x,\pi):=\sigma(x,\pi)\tilde{\gamma}(\pi)   :H_\pi\rightarrow H_\pi,
\end{equation} and 
\begin{equation*}
  \sigma^1(x,\pi):=\sigma(x,\pi)-\sigma^0(x,\pi):\textnormal{Dom}(\sigma(x,\pi))\cap \textnormal{Dom}(\sigma^0(x,\pi))\supset H_\pi^\infty\rightarrow H_\pi.
\end{equation*}
If we denote by $A^j=\textnormal{Op}(\sigma^j),$  and $(A^jf)_{B(x_0,r)}=\textnormal{Average}(A^jf,B(x_0,r))$ for $j=0,1,$ we have,
\begin{align*}
     \frac{1}{|B(x_0,r)|}\int\limits_{B(x_0,r)}|Af(x)-&(Af)_{B(x_0,r)}|dx\\
  &   \leq\sum_{j=0,1}  \frac{1}{|B(x_0,r)|}\int\limits_{B(x_0,r)}|A^jf(x)-(A^jf)_{B(x_0,r)}|dx.
\end{align*}
To estimate the integral
\begin{equation*}
    I_0:=\frac{1}{|B(x_0,r)|}\int\limits_{B(x_0,r)}|A^0f(x)-(A^0f)_{B(x_0,r)}|dx,
\end{equation*}we will use the Mean Value Theorem (see \cite[page 119]{FischerRuzhanskyBook2015}). Indeed, observe that
\begin{align*}
    |A^0f(x)-A^0f(y)| &\leq C_{0}\sum_{k=1}^{\textnormal{dim}(G)}|y^{-1}x|^{\nu_j}\sup_{|z|\leq \eta|y^{-1}x|}|(X_kA^0f)(yz)|\\
    &\lesssim \sum_{k=1}^{\textnormal{dim}(G)}r^{\nu_j}\Vert X_kA^0f\Vert_{L^\infty(G)}.
\end{align*}
So, if $r\geq 1,$
\begin{equation}\label{nu1}
    |A^0f(x)-A^0f(y)| \lesssim r^{Q}\sup_{1\leq k\leq \dim(G)}\Vert X_kA^0f\Vert_{L^\infty(G)},
\end{equation}while for $0<r<1,$ we have
\begin{equation}\label{nu2}
    |A^0f(x)-A^0f(y)| \lesssim r^{\nu_{*}}\sup_{1\leq k\leq \dim(G)}\Vert X_kA^0f\Vert_{L^\infty(G)},\,\,\,\,\nu_*=\min_{1\leq j\leq \dim(G)}\{\nu_j\}.
\end{equation}

To estimate the $L^\infty$-norm of $X_kA^0f,$ let us observe that, in view of \eqref{formulasymbol}, the operator-valued symbol of $X_kA^0=\textnormal{Op}(\sigma'_k)$ is given by
\begin{equation*}
\sigma'_k(x,\pi):=   \pi(X_k)\sigma^0(x,\pi)+(X_k\sigma^0(x,\pi)).
\end{equation*}
Indeed, the Leibniz law gives
\begin{align*}
    X_kA^0f(x)&=\int\limits_{\widehat{G}}\textnormal{\textbf{Tr}}(X_k(\pi(x)\sigma^0(x,\pi))\widehat{f}(\pi))d\pi\\
    &=\int\limits_{\widehat{G}}\textnormal{\textbf{Tr}}([X_k(\pi(x))\sigma^0(x,\pi)+\pi(x)X_k\sigma^0(x,\pi))]\widehat{f}(\pi))d\pi.
\end{align*}
Because, $\pi(X_k)=\pi(x)^*X_k\pi(x),$ we have $X_k\pi(x)=\pi(x)\pi(X_k),$ and we obtain
\begin{align*}
     X_kA^0f(x)&=\int\limits_{\widehat{G}}\textnormal{\textbf{Tr}}([\pi(x)\pi(X_k)\sigma^0(x,\pi)+\pi(x)X_k\sigma^0(x,\pi))]\widehat{f}(\pi))d\pi.
\end{align*} By using a suitable partition of the unity we will decompose the operator $\sigma'_{k}(x,\pi)$ as follows:
\begin{equation*}
    \sigma'_{k}(x,\pi)=\sum_{j=1}^\infty\rho_{j,k}(x,\pi).
\end{equation*}
To construct the family of operators $\rho_{j,k}(x,\pi)$ we will proceed as follows. We choose a smooth real  function $\eta$ satisfying $\eta(t)\equiv 1$ for $|t|\leq 2^{-\nu}$ and $\eta(t)\equiv 0$ for $|t|\geq 2^{-\nu+1}.$ Set
\begin{equation*}
    \rho(t)=\eta(\frac{t}{2})-\eta(t).
\end{equation*} On the support of $\rho,$  $t\in \textnormal{supp}\rho $ implies that $t\sim 2^{-\nu}.$ One can check that 
\begin{equation*}
    1=\eta(t^{\nu})+\sum_{j=1}^\infty\rho(2^{-j\nu}t^\nu),\,\,\,\,\textnormal{ for all }t\in \mathbb{R}.
\end{equation*} 
Indeed, 
\begin{equation*}
    \eta(t^\nu)+\sum_{j=1}^\ell\rho(2^{-j\nu}t^{\nu})=\eta(t^\nu)+\sum_{j=1}^\ell\eta(2^{-j\nu+\nu}t^{\nu})-\eta(2^{-j\nu}t^{\nu})=\eta(2^{-\ell\nu+\nu}t^\nu)\rightarrow 1,\,\,\ell\rightarrow\infty.
\end{equation*}
For $t^\nu=r^{-\nu}\lambda,$ we have
\begin{equation*}
   1=\eta(r^{-\nu}\lambda)+\sum_{j=1}^\infty\rho(2^{-j\nu}r^{-\nu}\lambda),\,\,\,\,\textnormal{ for all }\lambda\in \mathbb{R}.
\end{equation*}
We can assume that $\lambda=0$ is an isolated point of the spectrum of $\mathcal{R},$ (see  Geller  \cite{Geller1983} or \cite[Section 3.2.8]{FischerRuzhanskyBook2015}). If $\textnormal{Spect}(\mathcal{R})\subset (\lambda_{R},\infty)$ with  $\lambda_{\mathcal{R}}>0,$ we have that
the spectral theorem implies,
\begin{align*}
    &I_{H_\pi}\equiv\\
    &\int\limits_{0}^{\infty}(\eta(\frac{r^{-\nu}}{\lambda^\nu_{\mathcal{R}}}\lambda)+\sum_{j=1}^\infty\rho(2^{-j\nu}\frac{r^{-\nu}}{\lambda^\nu_{\mathcal{R}}}\lambda))dE_{\pi(\mathcal{R})}(\lambda)\equiv \eta(\frac{r^{-\nu}}{\lambda^\nu_{\mathcal{R}}}\pi(\mathcal{R}))+\sum_{j=1}^\infty\rho(2^{-j\nu}\frac{r^{-\nu}}{\lambda^\nu_{\mathcal{R}}}\pi(\mathcal{R}))),
\end{align*} where the convergence of the operator series to the identity operator $I_{H_\pi}$ is understood in the sense of the strong topology on $\mathscr{L}(H_\pi),$ the set of linear operators on $H_{\pi}.$ This means that, for every $v,w\in H_{\pi}$ we have
\begin{equation*}
    ( v,w )_{H_\pi}= ( \eta(\frac{r^{-\nu}}{\lambda^\nu_{\mathcal{R}}}\pi(\mathcal{R}))v,w)_{H_\pi}+\sum_{j=1}^\infty(\rho(2^{-j\nu}\frac{r^{-\nu}}{\lambda^\nu_{\mathcal{R}}}\pi(\mathcal{R})))v,w)_{H_\pi}.
\end{equation*}
Because $\textnormal{supp}(\eta)\subset [1,\infty),$ if $r\lambda\leq 1,$ we have $\eta(r\lambda)\equiv 0$ and 
\begin{equation*}
    I_{H_\pi}\equiv \int\limits_{\lambda^\nu_{\mathcal{R}}}^{\infty}\sum_{j=1}^\infty\rho(2^{-j\nu}\frac{r^{-\nu}}{\lambda^\nu_{\mathcal{R}}}\lambda)dE_{\pi(\mathcal{R})}(\lambda)\equiv\sum_{j=1}^\infty\rho(2^{-j\nu}\frac{r^{-\nu}}{\lambda^\nu_{\mathcal{R}}}\pi(\mathcal{R}))),
\end{equation*}where the convergence of the operator series to the identity operator $I_{H_\pi}$ is understood in the sense of the strong topology on $\mathscr{L}(H_\pi).$ In view of \eqref{gamasupport}, $\textnormal{supp}(\tilde{\gamma})\subset \{t:|t|\leq 1\},$ we have
\begin{equation*}
    \sigma'_{k}(x,\pi)=\sum_{j=1}^\infty\sigma'_{k}(x,\pi)\rho(2^{-j\nu}\frac{r^{-\nu}}{\lambda^\nu_{\mathcal{R}}}\pi(\mathcal{R}))).
\end{equation*} We define
\begin{equation*}
    \rho_{j,k}(x,\pi):=\sigma'_{k}(x,\pi)\rho(2^{-j\nu}\frac{r^{-\nu}}{\lambda^\nu_{\mathcal{R}}}\pi(\mathcal{R}))):H_\pi\rightarrow H_\pi.
\end{equation*}
Because  $\textnormal{supp}(\rho)\subset[1,4] ,$ for every $j,$ the support of the function $\rho_{j}(\lambda):= \rho(\frac{r^{-\nu}}{\lambda^\nu_{\mathcal{R}}}2^{-j\nu}\lambda)$ satisfies 
\begin{equation*}
    \textnormal{supp}(\rho_j)\subset \{\lambda: 1\leq \frac{r^{-\nu}}{\lambda^\nu_{\mathcal{R}}}2^{-j\nu}\lambda\leq 4 \}=\{\lambda: 2^{j\nu}\leq \frac{r^{-\nu}}{\lambda^\nu_{\mathcal{R}}}\lambda\leq 2^{j\nu+2)} \}.
\end{equation*}
So, in the support of $\rho_j$ we have $\lambda\approx \lambda^\nu_{\mathcal{R}}2^{j\nu}r^{\nu}.$ If we use both, Lemma \ref{FundamentallemmaI} and Lemma \ref{FundamentallemmaII}, we have that
\begin{align*}
    &\Vert  \textnormal{Op}((x,\pi)\mapsto \rho_{j,k}(x,\pi))\Vert_{\mathscr{B}(L^\infty(G))} \\
    &\leq C \left(\sup_{\pi\in\widehat{G}, [\alpha_1] \leq \ell} \Vert [\Delta^{\alpha_1} \sigma'_{k,j}(x,\pi)\rho((2^{-j\nu}\frac{r^{-\nu}}{\lambda^\nu_{\mathcal{R}}}\cdot\pi)(\mathcal{R})))]\pi((1+\mathcal{R})^{\frac{ \frac{Q(1-\rho)}{2}+\rho [\alpha_1]  }{\nu}})  \Vert_{\textnormal{op}}     \right)\\
    &\lesssim  \Vert \sigma\Vert_{\ell,S^{m}_{\rho,\delta}}(r^{-\nu}2^{-j\nu})^{\frac{\ell}{\nu}}=  \Vert \sigma\Vert_{\ell,S^{m}_{\rho,\delta}}(r^{-1}2^{-j})^{\ell},
\end{align*} where in the last line we have used Lemma \ref{compare}.
In view of the inequality (see \eqref{nu1} and \eqref{nu2})
\begin{equation*}
    |A^0f(x)-A^0f(y)| \lesssim r^{\nu(r)}\sup_{1\leq k\leq \dim(G)}\Vert X_kA^0f\Vert_{L^\infty(G)},
\end{equation*} where $\nu(r)=\nu_{*}:=\min_{1\leq j\leq \dim(G)}\{\nu_j\},$   for $0\leq r\leq 1,$ and $\nu(r)=Q:=\sum_{j=1}^n\nu_j,$ for $r\geq 1,$ we have
\begin{align*}
     I_0:&=\frac{1}{|B(x_0,r)|}\int\limits_{B(x_0,r)}|A^0f(x)-(A^0f)_{B(x_0,r)}|dx\leq r^{\nu(r)}\sup_{1\leq j\leq \dim(G)}\Vert X_kA^0f\Vert_{L^\infty(G)}\\
     &= r^{\nu(r)} \sup_{1\leq k\leq \dim(G)}\Vert \textnormal{Op}(\sigma'_{k})f\Vert_{L^\infty(G)} \leq r^{\nu(r)}\sup_{1\leq k\leq \dim(G)}\sum_{j=1}^\infty\Vert \textnormal{Op}(\rho_{j,k})f\Vert_{L^\infty(G)}.
\end{align*} By using Lemma \ref{FundamentallemmaII} with $\ell:=\nu(r),$ $s=r^{-\nu}2^{-j\nu}\lambda_{\mathcal{R}}^{-\nu},$ and $\varepsilon=\frac{Q(1-\rho)}{2},$ we have that $r^{\nu(r)-\ell}=1,$ and for $m=Q(1-\rho)/2,$ we have

\begin{align*}
   I_0 &  \lesssim r^{\nu(r)} \sup_{1\leq k\leq \dim(G)}\Vert \textnormal{Op}(\sigma'_{k})f\Vert_{L^\infty(G)} \lesssim r^{\nu(r)}\sup_{1\leq k\leq \dim(G)}\sum_{j=1}^\infty r^{-\ell}2^{-j\ell} \Vert \sigma\Vert_{\ell,S^{m}_{\rho,\delta}}\Vert f\Vert_{L^\infty(G)}\\
   &=r^{\nu(r)-\ell}\sup_{1\leq k\leq \dim(G)}\sum_{j=1}^\infty 2^{-j\ell} \Vert \sigma\Vert_{\ell,S^{m}_{\rho,\delta}}\Vert f\Vert_{L^\infty(G)}\\
   &=\sup_{1\leq k\leq \dim(G)}\sum_{j=1}^\infty 2^{-j\ell} \Vert \sigma\Vert_{\ell,S^{m}_{\rho,\delta}}\Vert f\Vert_{L^\infty(G)}\\
     &\lesssim  \Vert \sigma\Vert_{\ell,S^{m}_{\rho,\delta}}\Vert f\Vert_{L^\infty(G)}.
\end{align*} Consequently,
\begin{equation*}
    \sup_{r>0} \frac{1}{|B(x_0,r)|}\int\limits_{B(x_0,r)}|A^0f(x)-(A^0f)_{B(x_0,r)}|dx\leq C \Vert \sigma\Vert_{\ell,S^{m}_{\rho,\delta}}\Vert f\Vert_{L^\infty(G)}.
\end{equation*} In order to obtain a similar $L^\infty(G)$-$BMO(G)$ estimate for $A^{1},$ we will proceed as follows. Let $\phi$ be a smooth function compactly supported in $B(x_0,2r)$ satisfying 
\begin{equation*}
    \phi(x)=1, \textnormal{  for  }\,\,x\in B(x_0,r),\textnormal{ and }0\leq \phi\leq 10.
\end{equation*}
Note that, as in the proof of Lemma \ref{LemmaJulio} we have that,
\begin{equation}\label{doubling}
      \Vert \phi\Vert_{L^2(G)}\leq 10|B(x_0,2r)|^{\frac{1}{2}}\leq 10C|B(x_0,r)|^{\frac{1}{2}}.
\end{equation} Taking into account that
\begin{align*}
    \frac{1}{|B(x_0,r)|}\int\limits_{B(x_0,r)}|A^1f(x)-(A^1f)_{B(x_0,r)}|dx\leq \frac{2}{|B(x_0,r)|}\int\limits_{B(x_0,r)}|A^1f(x)|dx,
\end{align*}
we will estimate the right-hand side. Indeed, taking into account that $\phi=1$ on $B(x_0,r),$ let us observe that
\begin{align*}
 &\frac{1}{|B(x_0,r)|} \int\limits_{B(x_0,r)}|A^1f(x)|dx=  \frac{1}{|B(x_0,r)|}\int\limits_{B(x_0,r)}|\phi(x)A^1f(x)|dx\\
  &\leq  \frac{1}{|B(x_0,r)|}\int\limits_{B(x_0,r)}|A^1[\phi f](x)|dx+\frac{1}{|B(x_0,r)|} \int\limits_{B(x_0,r)}|[M_{\phi},A^1 ] f(x)|dx\\
 :&=I+II,
\end{align*}where $M_\phi$ is the multiplication  operator by $\phi.$ 
In order to estimate $I,$ we can use Lemma \ref{LemmaJulio} with $T=A^1,$ $\tau=\sigma^1,$ and $\varepsilon=\frac{Q(1-\rho)}{2},$ in order to claim that
\begin{equation}
   I:=  \frac{1}{|B(x_0,r)|}\int\limits_{B(x_0,r)}|A^1[\phi f](x)|dx\leq  {  C\Vert \sigma_{A^1L}\Vert_{k,S^{0}_{\rho,\delta}} }\Vert f\Vert_{L^\infty(G)},
\end{equation}for some $k\in \mathbb{N}_0.$

Now, we will estimate $II.$ For this, observe that the symbol of $[M_{\phi},A^1 ]=\textnormal{Op}(\theta) ,$ is given by
\begin{equation}\label{theta}
    \theta(x,\pi)=\int\limits_{G}(\phi(x)-\phi(xy^{-1}))k_{x}(y)\pi(y)^{*}dy,
\end{equation}
where $x\mapsto k_{x},$ is the right-convolution kernel of $A^1.$ The proof of equality \eqref{theta}  is the same as in the case of compact Lie groups (see \cite[page 554]{DelgadoRuzhansky2019}).
Using the Taylor expansion we obtain
\begin{equation*}
    \phi(xy^{-1})=\phi(x)+\sum_{|\alpha|=1}(X_{x}^{\alpha}\phi)(x)\tilde{q}_{\alpha}(y),
\end{equation*}
where, every  $\tilde{q}_\alpha$ is a smooth function vanishing with order $1$ at $e_G.$
So, we can write
\begin{equation*}
    \theta(x,\pi)=\sum_{ |\alpha| =1}{X_{x}^{\alpha}\phi(x)}\Delta_{ \tilde{q}_{\alpha}} \theta(x,\pi) .
\end{equation*}
By using the decomposition 
\begin{equation*}
    \theta(x,\pi) =\sum_{j=0}^\infty \theta_j(x,\pi),\,\,\,\theta_j(x,\pi)=\theta(x,\pi)\rho(2^{-j\nu}\frac{r^{-\nu}}{\lambda^\nu_{\mathcal{R}}}\pi(\mathcal{R}))),\,\,j\geq 1,
    \end{equation*} where $\theta_{0}(x,\pi)=\eta(\frac{r^{-\nu}}{\lambda^\nu_{\mathcal{R}}}\lambda)\theta(x,\pi),$
and from Lemma \ref{FundamentallemmaII}, we have the estimate
\begin{align*}
    \Vert\theta_j\Vert_{\ell',S^{-\frac{Q(1-\rho)}{2}}_{\rho,\delta}}\leq  \Vert \sigma\Vert_{\ell',S^{-\frac{Q(1-\rho)}{2}}_{\rho,\delta}}(2^{-j\nu}r^{-\nu})^{\frac{\ell'}{\nu}}=\Vert \sigma\Vert_{\ell',S^{-\frac{Q(1-\rho)}{2}}_{\rho,\delta}}(2^{-j}r^{-1})^{\ell'},
\end{align*} where we fix $\ell'\geq 1. $ The estimate
\begin{equation*}
    \frac{1}{|B(x_0,r)|} \int\limits_{B(x_0,r)}|[M_{\phi},A^1 ] f(x)|dx\leq \Vert [M_{\phi},A^1 ] f \Vert_{L^\infty(G)},
\end{equation*} and Lemma \ref{FundamentallemmaII} imply
\begin{align*}
     \frac{1}{|B(x_0,r)|} \int\limits_{B(x_0,r)}|[M_{\phi},A^1 ] f(x)|&\leq\sum_{j=1}^\infty\Vert\textnormal{Op}(\theta_j)f \Vert_{L^\infty(G)}\\
     &\lesssim \sum_{j=1}^\infty (r^{-1}2^{-j})^{\ell'} \Vert \sigma\Vert_{\ell',S^{-\frac{Q(1-\rho)}{2}}_{\rho,\delta}}\Vert f \Vert_{L^\infty(G)}.
\end{align*}
Thus, we obtain 
\begin{equation*}
    II:= \frac{1}{|B(x_0,r)|}\int\limits_{B(x_0,r)}|[M_\phi,A^1 ]f(x)|dx\leq {  C\Vert \sigma_{A^1L}\Vert_{\ell',S^{-\frac{Q(1-\rho)}{2}}_{\rho,\delta}} }\Vert f\Vert_{L^\infty(G)}.
\end{equation*} 
So, we have the estimate
\begin{equation*}
    \Vert  A\Vert_{\mathscr{B}(L^\infty(G),BMO(G))}\leq C\max\{ \Vert\sigma\Vert_{\ell,\,S^{-\frac{Q(1-\rho)}{2}}_{\rho,\delta}},    \Vert\sigma_{A(1+\mathcal{R})^{\frac{Q(1-\rho)}{2\nu}}}\Vert_{\ell,\,S^{0}_{\rho,\delta}}  \},
\end{equation*}for the operator norm of $A,$ provided that $\delta\leq \rho.$ Now, if $\delta\leq\rho,$ the symbolic calculus developed in \cite{FischerRuzhanskyBook2015} (see Theorem \ref{calculus}) implies that $A^{*}\in S^{-\frac{Q(1-\rho)}{2}}_{\rho,\delta}(G\times \widehat{G}).$ So, by the duality argument, and the duality between $H^1(G)$ and $BMO(G)$ in the homogeneous setting (see Christ and Geller \cite{ChristGeller1984}) we conclude
that
\begin{equation*}
   \Vert  A\Vert_{\mathscr{B}(H^1(G),L^1(G)))}\leq  C\max\{ \Vert\sigma^*\Vert_{\ell,\,S^{-\frac{Q(1-\rho)}{2}}_{\rho,\delta}},    \Vert\sigma_{A^*(1+\mathcal{R})^{\frac{Q(1-\rho)}{2\nu}}}\Vert_{\ell,\,S^{0}_{\rho,\delta}}  \},
\end{equation*}
Thus, the proof of Theorem \ref{LinftyBMOCardonaDelgadoRuzhansky} is complete.
\end{proof}

\subsection{$L^p$-boundedness for  pseudo-differential operators}
Now we will analyse the $L^p$-boundedness for pseudo-differential operators essentially in two cases. The first one is  the general $L^p$-boundedness  for operators in the class $S^{-\frac{Q(1-\rho)}{2}}_{\rho,\delta}(G\times \widehat{G}).$ In the  other one case, we will study the $L^p$-boundedness of pseudo-differential operators associated to the class  $S^{-\varepsilon}_{\rho,\delta}(G\times \widehat{G}),$ with $0<\varepsilon<\frac{Q(1-\rho)}{2},$ which as in Fefferman's theorem holds true only for suitable intervals centered at $p=2.$ In the first case we obtain the following theorem.
\begin{theorem}\label{Lp1CardonaDelgadoRuzhansky2}
Let $G$ be a graded Lie group of homogeneous dimension $Q.$ Let $A:C^\infty(G)\rightarrow\mathscr{D}'(G)$ be a pseudo-differential operator with symbol $\sigma\in S^{-m}_{\rho,\delta}(G\times \widehat{G} ),$  $m\geq Q(1-\rho)/2,$ $0\leq \delta\leq \rho\leq 1,$ $\delta\neq 1.$ Then $A=\textnormal{Op}(\sigma)$ extends to a bounded operator from $L^p(G)$ to $L^p(G)$ for all $1<p<\infty.$
\end{theorem} 
\begin{proof}
From Theorem \ref{LinftyBMOCardonaDelgadoRuzhansky}, for $0\leq \delta\leq \rho\leq 1,$ $\delta\neq 1,$ $A=\textnormal{Op}(\sigma)$ extends to a bounded operator from $L^\infty(G)$ to $BMO(G),$ and from the Calder\'on-Vaillancourt theorem, $A$ is also bounded on $L^2(G).$ So, the Fefferman-Stein interpolation theorem gives the $L^p$-boundedeness of $A$ for all $2\leq p< \infty.$ Now, if $0\leq \delta\leq\rho\leq 1,$ $\delta\neq 1,$ the adjoint operator $A^*\in S^{-\frac{Q(1-\rho)}{2}}_{\rho,\delta}(G\times \widehat{G}),$ extends to a bounded operator from $H^1(G)$ to $L^1(G),$ and the both, the $L^2$-boundededness of $A^*,$ the real interpolation and the  duality argument give the $L^p$-boundedeness of $A$ for all $1< p< 2.$
\end{proof}

Part (b) in Theorem \ref{MainTheorem} corresponds to the following $L^p$-boundedness theorem.
\begin{theorem}\label{partb}
Let $G$ be a graded Lie group of homogeneous dimension $Q.$ Let $A:C^\infty(G)\rightarrow\mathscr{D}'(G)$ be a pseudo-differential operator with symbol $\sigma\in S^{-m}_{\rho,\delta}(G\times \widehat{G} ),$ $0\leq \delta\leq \rho\leq 1,$ $\delta\neq 1.$ If  $1<p<\infty,$ then $A=\textnormal{Op}(\sigma)$ extends to a bounded operator from $L^p(G)$ to  $L^p(G)$ provided that \begin{equation*}
   m\geq m_{p}:=Q(1-\rho)\left|\frac{1}{p}-\frac{1}{2}\right|.
\end{equation*}
\end{theorem} 
\begin{proof} Let $0\leq \delta\leq \rho\leq  1,$ $\delta\neq 1,$ and $2\leq p<\infty.$ 
We will use the complex Fefferman-Stein interpolation theorem. We only need to prove the theorem for $m=m_p$ in view of the inclusion $S^{-m}_{\rho,\delta}(G\times \widehat{G})\subset S^{-m_p}_{\rho,\delta}(G\times \widehat{G})  $  for $m>m_p.$ Let us consider the complex family of operators indexed by $z\in \mathbb{C},$ $ \mathfrak{Re}(z)\in [0,1],$
\begin{equation*}
    T_{z}:=\textnormal{Op}(\sigma_{z}),\,\,\,\, \sigma_{z}(x,\pi):=e^{z^2}\sigma(x,\pi)(1+\pi(\mathcal{R}))^{\frac{m+\frac{Q(1-\rho)}{2}(z-1)}{\nu}}.
\end{equation*} The family of operators $\{T_{z}\},$ defines an analytic family of operator from $ \mathfrak{Re}(z)\in (0,1),$ (resp. continuous for $ \mathfrak{Re}(z)\in [0,1]$) into the algebra of bounded operators on $L^2(G).$  Let us observe that $\sigma_0(x,\pi)=\sigma(x,\pi)(1+\pi(\mathcal{R}))^{\frac{m-\frac{Q(1-\rho)}{2}}{\nu}},$ and $\sigma_{1}(x,\pi)=e\sigma(x,\pi)(1+\pi(\mathcal{R}))^{\frac{m}{\nu}}.$ Because $T_0$ is bounded from $L^\infty(G)$ into $BMO(G)$ and $T_1$ is bounded on $L^2(G),$ the Fefferman-Stein interpolation theorem implies that $T_t$ extends to a bounded operator on $L^p(G),$ for $p=\frac{2}{t}$ and all $0<t\leq 1.$ Because $0\leq m\leq\frac{Q(1-\rho)}{2}, $ there exist $t_0\in (0,1)$ such that $m=m_p=\frac{Q(1-\rho)}{2}(1-t_{0}).$ So, $T_{t_0}=e^{t_0^2}A$ extends to a bounded operator on $L^{\frac{2}{t_0}}.$ Now, we have two options for $p,$ indeed, $p\in [2,\frac{2}{t_0})$ or $p\in [\frac{2}{t_0},\infty).$ In both cases with real interpolation  between, the $L^2(G)$-boundedness and the  $L^{\frac{2}{t_0}}$-boundedness  of $A$ provide the $L^r(G)$-boundedness of $A$ for all $2\leq r\leq \frac{2}{t_0},$ and interpolating the $L^{\frac{2}{t_0}}(G)$-boundedness with the $L^\infty(G)$-$BMO(G)$ boundedness of $A$ we deduce the boundedneess of $A$ on $L^r(G)$ for all $\frac{2}{t_0}\leq r<\infty.$ So, $A$ extends to a bounded operator on $L^p(G)$ for all $2\leq p<\infty.$ The $L^p(G)$-boundedness of $A$ for $1<p\leq 2$ and $0\leq \delta\leq \rho\leq 1,$ $\delta\neq 1,$ now follows by the duality argument.
\end{proof}

\begin{remark}\label{remarkestimate}
Because, we have used the Calder\'on-Vaillacourt theorem in the graded setting (see Proposition 5.7.14 of \cite{FischerRuzhanskyBook2015}) as a crucial tool in the proof the local estimate \ref{LE} in Lemma \ref{LemmaJulio}, we improve the expected condition $\delta<\rho$ allowing the case $\rho=\delta$ on a arbitrary graded group $G,$ same as the Fefferman's theorem version by  C. Z. Li, and R. H. Wang,  \cite{RouhuaiChengzhang} for $G=\mathbb{R}^n.$  In particular we allow that case of operators of order zero an of $(1,\delta)$-type,  $0\leq \delta<1.$
\end{remark}

\begin{remark}
We found the critical  order $m_{p}=Q(1-\rho)|\frac{1}{p}-\frac{1}{2}|$ for the $L^p$-boundedness of pseudo-differential operators. For the especial case of spectral multipliers of the sub-Laplacian $\mathcal{L}$, on a stratified  Lie group this order can be relaxed, see e.g. Alexopoulos \cite{Alexopoulos1994}.
\end{remark}

\begin{remark}[Global H\"ormander classes on the Heisenberg group]\label{MainRemark}
Let us mention how  the H\"ormander classes $S^m_{\rho,\delta}(\mathbb{H}^n\times \widehat{\mathbb{H}}^n),$ $m\in \mathbb{R},$ $0\leq \delta\leq \rho\leq  1,$ $\delta\neq 1,$   can be characterized in terms of the Shubin classes (see e.g. \cite{FischerRuzhansky2014} or \cite[Chapter 6]{FischerRuzhanskyBook2015}). Indeed, using  the fact that $ \widehat{\mathbb{H}}^n\sim \mathbb{R}^{*}:= \mathbb{R}^{+}\cup \mathbb{R}^{-}, $ every (global) pseudo-differential operator $A$  with symbol $\sigma\in S^m_{\rho,\delta}(\mathbb{H}^n\times \widehat{\mathbb{H}}^n),$ defined via
\begin{equation*}
    Af(x)=\int\limits_{\mathbb{R}^*}\textnormal{\textbf{Tr}}(\pi_\lambda(g)\sigma(g,\pi_\lambda)\widehat{f}(\pi_{\mathbb{\lambda}}))d\lambda_n,
\end{equation*}
gives rise to a  parametrized family of densely defined operators $\sigma(g,\pi_{\lambda}),$ $g\in \mathbb{H}^n,$ and $\lambda\in \mathbb{R}^{*},$ defined on the Schwartz space $\mathscr{S}(\mathbb{R}^n).$ Except for a set of zero measure, every operator $\sigma(g,\pi_{\lambda})$ is a pseudo-differential operator on $\mathbb{R}^n,$ and in terms of the Weyl-quantization, it has a symbol $\sigma_{(g,\lambda)}\in C^\infty(\mathbb{R}^n\times \mathbb{R}^n ),$ such that $\sigma(g,\pi_{\lambda})=\textnormal{Op}^{w}[\sigma_{(g,\lambda)}],$ where 
\begin{equation*}
    \textnormal{Op}^{w}[\sigma_{(g,\lambda)}]h(x)=(2\pi)^{-n}\int\limits_{\mathbb{R}^n}\int\limits_{\mathbb{R}^n}e^{i(u-v)\cdot \xi}\sigma_{(g,\lambda)}(\xi,\frac{1}{2}(u+v))h(v)dvd\xi,
\end{equation*} for every $h\in \mathscr{S}(\mathbb{R}^n).$ Theorem 6.5.1 in \cite{FischerRuzhanskyBook2015} establishes the following equivalence,
\begin{equation*}
    \sigma\in S^m_{\rho,\delta}(\mathbb{H}^n\times \widehat{\mathbb{H}}^n),\textnormal{  if and only if,  }\partial_{u}^{\alpha_3} \partial_{\lambda}^{\alpha_1}\partial_{\xi}^{\alpha_2}  X_{g}^{\beta}\sigma_{(g,\lambda)}\in \Sigma_{\rho}^{m-2\rho |\alpha| +\delta  [\beta]  }(\mathbb{R}^n),
\end{equation*}
where $\Sigma_{\rho}^{m'}(\mathbb{R}^n),$ denotes the Shubin class of type $\rho$ and of order $m'\in \mathbb{R}^n,$ (see e.g. \cite[Chapter 6]{FischerRuzhanskyBook2015}) which can be defined by those symbols $a\in C^{\infty}(\mathbb{R}^n\times \mathbb{R}^n )$ satisfying,
\begin{equation}\label{Shubin}
  |\partial_{x}^{\beta}\partial_\xi^{\alpha}a(x,\xi)|\leq C_{\alpha,\beta}\langle x,\xi\rangle^{m'-\rho( |\alpha| + [\beta] )},
\end{equation}
where $\langle x, \xi\rangle:=(1+|x|^2+|\xi|^2)^{\frac{1}{2}}.$
By using that the homogeneous dimension of the Heisenberg group $\mathbb{H}^n$ is $Q=2n+2,$ in terms of the Shubin classes, Theorem \ref{MainTheorem} implies the following result.
\end{remark}
\begin{corollary}\label{MainCorollary}
 Let $A:C^\infty(\mathbb{H}^{n})\rightarrow\mathscr{D}'(\mathbb{H}^n)$ be a continuous linear operator with symbol $\sigma,$ defined by $\sigma(g,\pi_{\lambda})=\textnormal{Op}^{w}[\sigma_{(g,\lambda)}].$ Let us assume that $\partial_{u}^{\alpha_3} \partial_{\lambda}^{\alpha_1}\partial_{\xi}^{\alpha_2}   X_{g}^{\beta}\sigma_{(g,\lambda)}\in \Sigma_{\rho}^{-m-2\rho [\alpha] +\delta [\beta] }(\mathbb{R}^n),$ $0\leq \delta\leq \rho\leq 1,$ $\delta\neq 1,$ for every $\alpha$ and $\beta$ in $\mathbb{N}_0^n.$   Then,
\begin{itemize}
    \item{\textnormal{(a)}} if $m=(n+1)(1-\rho),$    then $A$ extends to a bounded operator from $L^\infty(\mathbb{H}^n)$ to $ BMO(\mathbb{H}^n).$ Moreover,   $A$ also admits a bounded extension from the Hardy space $H^1(\mathbb{H}^n)$ to $L^1(\mathbb{H}^n)$.
   \item{\textnormal{(b)}} If $m\geq m_{p}:= (2n+2)(1-\rho)\left|\frac{1}{p}-\frac{1}{2}\right|,$ then $A$ extends to a bounded operator on $ L^p(\mathbb{H}^n).$ 
\end{itemize}
\end{corollary}

\subsection{Boundedness for pseudo-differential operators on Sobolev and Besov spaces} Now, we will present some results for the boundedness of operators on Besov and Sobolev spaces by using the $L^p$-estimates proved in Theorem \ref{MainTheorem}. We recall that the  Sobolev space $L^{p}_{s}(G)$ is  defined by the norm (see \cite[Chapter 4]{FischerRuzhanskyBook2015})
\begin{equation}\label{L2ab}
    \Vert f \Vert_{L^{p}_{s}(G)}=\Vert (1+\mathcal{R})^{\frac{s}{\nu}}f\Vert_{L^p(G)},
\end{equation} for $1<p<\infty$ and $s\in \mathbb{R}.$
\begin{theorem}\label{Sobolevtheorem}
Let $G$ be a graded Lie group of homogeneous dimension $Q.$ Let $A:C^\infty(G)\rightarrow\mathscr{D}'(G)$ be a pseudo-differential operator with symbol $\sigma\in S^{-m}_{\rho,\delta}(G\times \widehat{G} ),$ $0\leq \delta\leq \rho\leq 1,$ $\delta\neq 1.$ Then, if $1<p<\infty,$ and $s\in \mathbb{R},$ the operator $A=\textnormal{Op}(\sigma)$ extends to a bounded operator from $L^p_{s}(G)$ to  $L^p_{s}(G)$ provided that 
\begin{equation*}
   m\geq m_{p}:=Q(1-\rho)\left|\frac{1}{2}-\frac{1}{p}\right|.
\end{equation*}
\end{theorem}
\begin{proof}
For the proof, we only need to show that there exists a positive constant $C>0$ such that 
\begin{equation}
    \Vert(1+\mathcal{R})^{\frac{s}{\nu}} Af\Vert_{L^p(G)}\leq C\Vert (1+\mathcal{R})^{\frac{s}{\nu}}f\Vert_{L^p(G)},
\end{equation}for every $f\in \mathscr{S}(G).$ Because $(1+\mathcal{R})^{\frac{s}{\nu}}:\mathscr{S}(G)\rightarrow \mathscr{S}(G),$ is an isomorphism of Frechet spaces we only need to prove that the estimate
\begin{equation}\label{eqisob}
    \Vert(1+\mathcal{R})^{\frac{s}{\nu}} A(1+\mathcal{R})^{-\frac{s}{\nu}}f\Vert_{L^p(G)}\leq C\Vert f\Vert_{L^p(G)},
\end{equation}holds true. However, from the global calculus developed in \cite{FischerRuzhanskyBook2015}, we have that $(1+\mathcal{R})^{\frac{s}{\nu}} A(1+\mathcal{R})^{-\frac{s}{\nu}}\in S^{-m}_{\rho,\delta}(G\times \widehat{G}).$ So, the estimate \ref{eqisob} now follows from Theorem \ref{partb}.
\end{proof}
Now, we will use the boundedness of operators on Sobolev spaces to deduce similar properties on Besov spaces. We refer the reader to \cite{CR} for the definition of Besov spaces on graded Lie groups and well as some of their properties.
\begin{theorem}\label{Besovtheorem}
Let $G$ be a graded Lie group of homogeneous dimension $Q.$ Let $A:C^\infty(G)\rightarrow\mathscr{D}'(G)$ be a pseudo-differential operator with symbol $\sigma\in S^{-m}_{\rho,\delta}(G\times \widehat{G} ),$ $0\leq \delta\leq \rho\leq1,$ $\delta\neq 1.$ Then $A=\textnormal{Op}(\sigma)$ extends to a bounded operator from $B^s_{p,q}(G)$ to  $B^s_{p,q}(G)$ for 
\begin{equation*}
   m\geq m_{p}:=Q(1-\rho)\left|\frac{1}{2}-\frac{1}{p}\right|,
\end{equation*}for all $0<q\leq \infty,$  $s\in\mathbb{R},$  and  $1<p<\infty.$
\end{theorem}
\begin{proof}
 We will use the real interpolation of Banach spaces to deduce the Besov boundedness of $A.$ If  $s\in \mathbb{R},$ Theorem \ref{Sobolevtheorem} shows that  $A$ extends to a bounded operator from $L^{p}_{s}(G)$ into $L^{p}_{s}(G)$ for every $1<p<\infty.$ In particular, if $1<p_0<p_1<\infty$ and $\theta\in (0,1)$ satisfies $1/p=\theta/p_0+(1-\theta)/p_1,$ then from the boundedness of the following bounded extensions of $A,$
 \begin{equation}
     A:L^{p_0}_{s}(G)\rightarrow L^{p_0}_{s}(G),\,\,\,\,A:L^{p_1}_{s}(G)\rightarrow L^{p_1}_{s}(G),
 \end{equation} and by the real interpolation of Banach spaces, we deduce that
 \begin{equation}
     A:(L^{p_0}_{s}(G),L^{p_1}_{s}(G))_{(\theta,q)}\rightarrow (L^{p_0}_{s}(G),L^{p_1}_{s}(G))_{(\theta,q)},\,\,\,\,0<q<\infty.
 \end{equation} From Theorem 3.2 of \cite{CR}, $(L^{p_0}_{s}(G),L^{p_1}_{s}(G))_{(\theta,q)}=B^{s}_{p,q}(G)$ for every $s\in \mathbb{R},$ and we conclude that  $A$ extends to a bounded operator from $B^{s}_{p,q}(G)$ into $B^{s}_{p,q}(G).$ 
 \end{proof}

\begin{remark}
Theorem \ref{Besovtheorem} extends to the non-commutative setting and in the case $0\leq \delta\leq \rho\leq  1,$ $\delta\neq 1,$ the classical Besov estimates for H\"ormander classes. We refer the reader e.g.  to Bourdaud \cite{Bourdaud} and Park \cite{Park} for details on the subject in the case of $\mathbb{R}^n.$
\end{remark}

\subsection{Local H\"ormander classes on graded Lie groups}
 Let $0\leq \delta,\rho\leq 1,$ and let $\mathcal{R}$ be a positive Rockland operator of homogeneous degree $\nu>0.$ If $m\in \mathbb{R},$ we say that the symbol $\sigma\in L^\infty_{a,b}(\widehat{G}), $ where $a,b\in\mathbb{R},$ belong locally  to the $(\rho,\delta)$-H\"ormander class of order $m,$ $S^m_{\rho,\delta,\textnormal{loc}}(G\times \widehat{G}),$ if for all $\gamma\in \mathbb{R},$ and  for every compact subset $K\subset G,$ the following conditions
\begin{equation*}
   p_{\alpha,\beta,\gamma,m,K}(\sigma)= \operatornamewithlimits{ess\, sup}_{(x,\pi)\in K\times \widehat{G}}\Vert \pi(1+\mathcal{R})^{\frac{\rho [\alpha] -\delta [\beta] -m-\gamma}{\nu}}[X_{x}^\beta \Delta^{\alpha}\sigma(x,\pi)] \pi(1+\mathcal{R})^{\frac{\gamma}{\nu}}\Vert_{\textnormal{op}}<\infty,
\end{equation*}
hold true for all $\alpha$ and $\beta$ in $\mathbb{N}_0^n.$ The resulting class $S^m_{\rho,\delta,\textnormal{loc}}(G\times \widehat{G}),$ does not depend on the choice of the Rockland operator $\mathcal{R}.$ 
These local versions of H\"ormander classes also provide a symbolic calculus closed under compositions, adjoints, and existence of parametrices. The following is a localised version of Theorem \ref{partb}.
\begin{theorem}\label{LocalSob}
Let $G$ be a graded Lie group of homogeneous dimension $Q.$  Let $A:C^\infty(G)\rightarrow\mathscr{D}'(G)$ be a pseudo-differential operator with symbol $\sigma\in S^{-m}_{\rho,\delta,\textnormal{loc}}(G\times\widehat{G} ).$ 
 If $m\geq m_{p}:= Q(1-\rho)\left|\frac{1}{p}-\frac{1}{2}\right|,$ then $A$ extends to a bounded operator from  $ L^p_{\textnormal{comp}}(G)$ to $ L^p_{\textnormal{loc}}(G),$  provided that  $0\leq \delta\leq \rho\leq  1,$ $\delta\neq 1,$ and $ 1<p<\infty.$
\end{theorem}
\begin{proof}
Let us assume that $f\in L^p_{\textnormal{loc}}(G).$ Then, for every compactly supported function $\omega\in C^\infty_0(G),$ $f\omega\in L^p(G).$  If $K$ is a compact subset of $G$ and $\omega'\in C^\infty_0(G)$ has compact support in $K,$ then the symbol $\omega'\sigma:=\{\omega'(x)\sigma(x,\pi)\}_{(x,\pi)\in G\times \widehat{G}}\in S^{-m}_{\rho,\delta}(G\times \widehat{G} ). $ So, in view of Theorem \ref{partb}, for some positive constant $C_{K}>0,$ (depending on $K,$ $\omega$ and $\omega'$) we have
\begin{equation*}
    \Vert \omega'A(\omega f)\Vert_{L^p(G)}\leq C_{K}\Vert \omega f \Vert_{L^p(G)},
\end{equation*} provided that $m\geq Q(1-\rho)\left|\frac{1}{p}-\frac{1}{2} \right|.$ Thus, $A$ extends to a bounded operator from  $ L^p_{\textnormal{comp}}(G)$ to $ L^p_{\textnormal{loc}}(G),$ and this concludes the proof.
\end{proof}

\begin{remark}\label{losswin} It is convenient to define the local Sobolev spaces for $s\in \mathbb{R},$ and $1<p<\infty,$ by 
\begin{equation*}
L^{p}_{s}(G,loc)=\{f\in \mathcal{D}'(G):\phi\cdot f\in L^{p}_{s}(G),\textnormal{ for all }\phi\in C^{\infty}_{0}(G)\}.
\end{equation*}  In view of the embedding (see \cite[Page 240]{FischerRuzhanskyBook2015})
\begin{equation}\label{enb}
L^{p}_{\frac{s}{\nu_1}}(G,loc)\subset L^p_{s}(\mathbb{R}^n,loc)\subset L^{p}_{\frac{s}{\nu_n}}(G,loc),
\end{equation}where $\nu_{1}\leq \nu_2\leq \dots \leq \nu_{n}$ are the weights associated to the homogeneous structure of $G,$ Theorem \ref{LocalSob}, implies that if $A:C^\infty(G)\rightarrow\mathscr{D}'(G)$ is a pseudo-differential operator with compactly supported with respect to $x,$ symbol $\sigma\in S^{-m}_{\rho,\delta,\textnormal{loc}}(G\times \widehat{G} ),$ then for $m\geq m_{p}:= Q(1-\rho)\left|\frac{1}{p}-\frac{1}{2}\right|,$ the linear operator  $A$ extends to a bounded operator from  $ L^{p}_{s}(G,loc)$ to $ L^{p}_{s}(G,loc),$ for any $1<p<\infty.$ Now, in view of \eqref{enb}, we also have that  $A$ extends to a bounded operator from  $ L^{p}_{\frac{s}{\nu_1}}(G,loc)$ to $ L^{p}_{s}(\mathbb{R}^n,loc),$ and from  $L^{p}_{s}(\mathbb{R}^n,loc) $ to $ L^{p}_{\frac{s}{\nu_n}}(G,loc).$ In the first situation we gain regularity with order $s-\frac{s}{\nu_1}=s(1-\frac{1}{\nu_1})\geq 0$ and in the other one, we lose regularity with order $\frac{s}{\nu_n}-s=s(\frac{1}{\nu_n}-1)\leq 0.$ In terms of this discussion, let us note that for $G=\mathbb{R}^n,$ we have $s(1-\frac{1}{\nu_1})=s(\frac{1}{\nu_n}-1)= 0,$ because in this case $\nu_1=\nu_n=1.$ Finally, for every $s\in\mathbb{R},$ $A: L^{p}_{s}(\mathbb{R}^n,loc)\rightarrow L^{p}_{\frac{s\nu_1}{\nu_n}}(\mathbb{R}^n,loc)$ extends to a bounded operator showing that in local Sobolev spaces on $\mathbb{R}^n$ we lose regularity.
\end{remark}

\begin{remark}\label{losswin2} Local Besov spaces for $s\in \mathbb{R},$ $0<q\leq \infty,$ and $1<p<\infty,$ are defined by 
\begin{equation*}
B^s_{p,q}(G,loc)=\{f\in \mathcal{D}'(G):\phi\cdot f\in B^s_{p,q}(G),\textnormal{ for all }\phi\in C^{\infty}_{0}(G)\}.
\end{equation*}  In view of the embedding (see \cite[Page 404]{CR})
\begin{equation}\label{enb2}
B^{\frac{s}{\nu_1}}_{p,q}(G,loc)\subset B^{s}_{p,q}(\mathbb{R}^n,loc)\subset B^{\frac{s}{\nu_n}}_{p,q}(G,loc),
\end{equation} Theorem \ref{LocalSob} implies that an operator $A:C^\infty(G)\rightarrow\mathscr{D}'(G)$  with compactly supported symbol in $x,$ $\sigma\in S^{-m}_{\rho,\delta,\textnormal{loc}}(G\times \widehat{G} ),$ where $m\geq m_{p}:= Q(1-\rho)\left|\frac{1}{p}-\frac{1}{2}\right|,$  extends to a bounded operator from  $ B^{s}_{p,q}(G,loc)$ to $ B^s_{p,q}(G,loc),$ for any $1<p<\infty,$ and $0<q\leq \infty.$ By using \eqref{enb2}, we deduce that  $A$ extends to a bounded operator from  $ B^{\frac{s}{\nu_1}}_{p,q}(G,loc)$ to $ B^{s}_{p,q}(\mathbb{R}^n,loc),$ and from  $B^{s}_{p,q}(\mathbb{R}^n,loc) $ to $B^{\frac{s}{\nu_n}}_{p,q}(G,loc).$ This analysis allows us to conclude that  for every $s\in\mathbb{R},$ $A: B_{p,q}^{s}(\mathbb{R}^n,loc)\rightarrow B_{p,q}^{\frac{s\nu_1}{\nu_n}}(\mathbb{R}^n,loc)$ extends to a bounded operator. As in the case of local Sobolev spaces we also lose regularity unless  $G=\mathbb{R}^n$ or $s=0.$
\end{remark}

\bibliographystyle{amsplain}

\end{document}